\documentclass[reqno]{amsart}

\usepackage{geometry}                 
\usepackage{graphicx}
\usepackage{amssymb,latexsym,amsmath}
\usepackage{amsfonts,amsthm}
\usepackage{amscd}
\usepackage{mathrsfs}
\usepackage{cite,upref}
\usepackage{eucal}
\usepackage{epstopdf}
\usepackage{amsgen}
\usepackage{xspace}
\usepackage{verbatim}

\newtheorem{theorem}{Theorem}[section]
\newtheorem{lemma}[theorem]{Lemma}
\newtheorem{proposition}[theorem]{Proposition}

\theoremstyle{definition}
\newtheorem{definition}[theorem]{Definition}

\theoremstyle{remark}

\numberwithin{equation}{section}

\newcommand{\gd}{\Delta}
\newcommand{\ggd}{\delta}
\newcommand{\inpt}[1]{\langle #1 \rangle}

\newcommand{\ms}{\mathscr}

\newcommand{\gw}{\Omega}
\newcommand{\ga}{\gamma}

\newcommand{\gl}{\lambda}
\newcommand{\gL}{\Lambda}

\newcommand{\om}{\omega}

\newcommand{\vp}{\varphi}

\newcommand{\pdr}{\partial}
\newcommand{\fn}[1]{\|#1\|_{L^{4}}^{4}}
\newcommand{\sn}[1]{\|#1\|_{L^6}^6}

\newcommand{\tup}{\textup}
\newcommand{\csg}{\{ S(t)\}_{t\geq0}}

\begin{document}

\title[THE EXTENDED BRUSSELATOR SYSTEM]{\textbf{Global Dissipative Dynamics of\\ the Extended Brusselator System}}

\author[Yuncheng You]{Yuncheng You$^{*}$} 
\address{$^{*}$Department of Mathematics and Statistics, University of South Florida, Tampa, Florida 33620, USA and 
Department of Applied Mathematics, Shanghai Normal University, Shanghai 200234, China}
\email{you@mail.usf.edu}
\thanks{The first author is partly supported by NSF Grant \#DMS-1010998.}

\author[Shengfan Zhou]{Shengfan Zhou$^{\dag}$}
\address{$^{\dag}$Department of Applied Mathematics, Shanghai Normal University, Shanghai 200234, China}
\email{sfzhou@shu.edu.cn}
\thanks{The second author is partly supported by the Leading Academic Discipline Project of Shanghai Normal University (No.DZL707), the National Natural Science Foundation of China under the Grant No.10771139, the Ministry of Education of China (No. 20080270002), and the Foundation of Shanghai Talented Persons (No.049).}

\subjclass[2000]{37L30, 35B40, 35B41, 35K55, 35K57, 35Q80}



\keywords{Brusselator system, dissipative dynamics, global attractor, absorbing set, asymptotic compactness, fractal dimension}

\begin{abstract}
The existence of a global attractor for the solution semiflow of the extended Brusselator system in the $L^2$ phase space is proved, which is a cubic-autocatalytic and partially reversible reaction-diffusion system with linear coupling between two compartments. The method of grouping and re-scaling estimation is developed to deal with the challenge in proving the absorbing property and the asymptotic compactness of this typical multi-component reaction-diffusion systems. It is also proved that the global attractor is an $(H, E)$ global attractor with the $L^\infty$ regularity and that the Hausdorff dimension and the fractal dimension of the global attractor are finite. The results and methodology can find many applications and further extensions  in complex biological and biochemical dynamical systems.
\end{abstract}

\maketitle



\section{\textbf{Introduction}}

Consider an extended Brusselator system consisting of cubic-autocatalytic and partially reversible reaction-diffusion equations with linear coupling between two compartments,
\begin{align}
	\frac{\pdr u}{\pdr t} &= d_1 \gd u + a - (b + k)u + u^2 v + D_1 (w - u) + N\vp, \label{equ} \\
	\frac{\pdr v}{\pdr t} &= d_2 \gd v + bu - u^2 v + D_2 (z - v), \label{eqv} \\
	\frac{\pdr \vp}{\pdr t} &= d_3 \gd \vp + ku - (\gl + N)\vp + D_3 (\psi - \vp),  \label{eqp} \\
	\frac{\pdr w}{\pdr t} &= d_1 \gd w + a - (b + k)w + w^2 z + D_1 (u - w)  + N\psi, \label{eqw} \\
	\frac{\pdr z}{\pdr t} &= d_2 \gd z + bw - w^2 z + D_2 (v - z),  \label{eqz} \\
	\frac{\pdr \psi}{\pdr t} &= d_3 \gd \psi + kw - (\gl + N)\psi + D_3 (\vp - \psi),  \label{eqs} 
\end{align}
for $(t, x) \in (0, \infty) \times \gw$, where $\gw \subset \Re^{n} (n \leq 3)$ is a bounded domain with a locally Lipschitz continuous boundary, and the coefficients $d_i's,  D_i's, (i = 1, 2, 3), a, b, k, \gl$, and $N$ are positive constants. Note that there is a linear coupling between the two subsystems of components $(u, v, \vp)$ and $(w, z, \psi)$, which can be interpreted as the interaction between two cells or compartments. Given the homogeneous Dirichlet boundary condition 
\begin{equation} \label{dbc}
	u(t, x) = v(t, x) = \vp(t,x) = w (t, x) = z (t, x) = \psi(t,x) = 0, \quad t > 0, \; x \in \partial \gw,
\end{equation}
and an initial condition
\begin{equation} \label{ic}
	\begin{split}
        u(0,x) &= u_0 (x), \; v(0, x) = v_0 (x), \; \vp(0,x) = \vp_0 (x),  \\
        w(0,x) &= w_0 (x), \; z(0, x) = z_0 (x), \; \psi(0,x) = \psi_0 (x), \quad x \in \gw, 
        \end{split}
\end{equation}
we shall study the asymptotic dynamics of the semiflow of the weak solutions to this initial-boundary value problem \eqref{equ} -- \eqref{ic}. 

Historically the Brussels school led by the renowned physical chemist and Nobel Prize laureate (1977), Ilya Prigogine, made remarkable contributions in the research of complexity of cubic autocatalytic reactions. The mathematical model signifying their seminal work was named as the Brusselator, which is originally a system of two ordinary differential equations shown below, cf.  \cite{PL68, AN75, sS94}.  The scheme of chemical reactions described by the Brusselator kinetics is
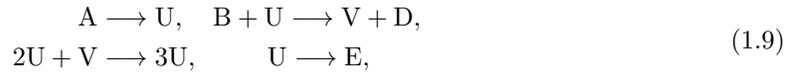
\begin{equation} \label{rc}
	\begin{split}
	\textup{A} & \longrightarrow \textup{U}, \quad \textup{B} + \textup{U} \longrightarrow \textup{V} + \textup{D}, \\
	2 \textup{U} + \textup{V} & \longrightarrow 3 \textup{U}, \quad \quad \; \; \textup{U} \longrightarrow \textup{E}, 
	\end{split}
\end{equation}
where \textup{A}, \textup{B}, \textup{U}, \textup{V}, \textup{D}, and \textup{E} are chemical reactants or products. Let $u(t, x)$ and $v(t, x)$ be the concentrations of \textup{U} and \textup{V}. Assume that the concentrations of the input compounds \textup{A} and \textup{B} are held constant during the reaction process, denoted by $a$ and $b$ respectively. Also assume that these elementary reactions are not reversible. Then by the Fick's law and the law of mass action, and through non-dimensionalization, one can formulate a system of two nonlinear reaction-diffusion equations called (diffusive) \textbf{Brusselator equations},
\begin{equation}  \label{br}
	\begin{split}
	\frac{\partial u}{\partial t} &= d_1 \gd u + a - (b + k)u + u^{2}v, \\
	\frac{\partial v}{\partial t} &= d_2 \gd v + bu - u^{2}v.
	\end{split}
\end{equation}

The known examples of autocatalytic reactions which can be modeled by the Brusselator equations include ferrocyanide-iodate-sulphite reaction, chlorite-iodide-malonic acid (CIMA) reaction, arsenite-iodate reaction, many enzyme catalytic reactions, cf. \cite{AO78, AN75, BD95}.

In 1993 two consecutive papers  \cite{LMOS93, jP93} discovered a variety of interesting self-replicating pattern formation associated with cubic-autocatalytic reaction-diffusion systems, respectively by an experimental approach and a numerical simulation approach. Since then numerous studies by mathematical and computational analysis have shown that the cubic-autocatalytic reaction-diffusion systems such as Brusselator equations and Gray-Scott equations \cite{GS83, GS84} exhibit rich spatial patterns (including but not restricted to Turing patterns) and complex bifurcations \cite{AO78, BGMLF00, BD95, gD87, DKZ97, FXWO08, IK06, JS08, aK83, PPG01, PW05, WW03} as well as interesting dynamics \cite{CQ07, ER83, KCD97, KEW06, yQ07, RR95, RPP97, Wet96, YZE04} on 1D or 2D domains.

For Brusselator equations and the other cubic-autocatalytic model equations of space dimension $n \leq 3$, however, we have not seen substantial advancing results in the front of global dynamics until recently \cite{yY07, yY08, yY09a, yY09b, yY10}. 

The extended Brusselator system in this work features the coupling of two compartment variables and the reversible reaction U $\leftrightarrows$ E in both compartments, where E is the main product being removed at a constant rate $\gl$. This system consisting of six coupled components with partial reversibility is closer to the biological network dynamics \cite{GLI07, hK02}. 

Multicell models in biological dynamics generically mean the coupled ODEs or PDEs with large number of unknowns representing component substances, which appear widely in the literature of systems biology as well as cell and molecular biology \cite{SM82, SS95, TCN01}.  

For this extended Brusselator system and many other biochemical activator-inhibitor reactions, the asymptotically dissipative condition in vector version
$$
	\lim_{|s| \to \infty} f (s) \cdot s \leq C,
$$
where $C$ is a nonnegative constant, is inherently not satisfied by the nonlinearity involving the opposite-signed autocatalytic terms and the opposite-signed coupling terms. This and the additional equations of variables $\vp$ and $\psi$ due to the partially reversing reactions cause substantial difficulties in \emph{a priori} estimation of all six components in proving the absorbing property and the asymptotic compactness of the solution semiflow. The novel mathematical feature in this paper is to overcome these obstacles by a method of \emph{grouping and re-scaling estimation}. This methodology will be potentially very useful in analyzing biological network dynamics of multicomponent reaction-diffusion systems.

We start with the formulation of an evolutionary equation associated with this system. Define the product Hilbert spaces as follows,
\begin{equation*}
    	H = [L^2 (\gw)]^6, \quad E =  [H_{0}^{1}(\gw)]^6, \quad  \Pi = [H_{0}^{1}(\gw) \cap H^{2}(\gw)]^6.
\end{equation*}
These six-component spaces are the phase spaces of different regularity for the component functions $u(t, \cdot), v(t, \cdot), \vp (t, \cdot), w(t, \cdot),z (t, \cdot)$ and $\psi(t, \cdot)$. The norm and inner-product of $H$ or any component space $L^2 (\gw)$ will be denoted by $\| \, \cdot \, \|$ and $\inpt{\,\cdot , \cdot\,}$, respectively. The norm of $L^p (\gw)$ or the product space $\mathbb{L}^p (\gw) = [L^p (\gw)]^6$ will be denoted by $\| \, \cdot \, \|_{L^{p}}$ if $p \ne 2$. By the Poincar$\Acute{e}$ inequality and the homogeneous Dirichlet boundary condition \eqref{dbc}, there is a constant $\ga > 0$ such that
\begin{equation} \label{pcr}
    \| \nabla \xi \|^2 \geq \ga \| \xi \|^2, \quad \textup{for} \;  \xi \in H_{0}^{1}(\gw) \; \textup{or} \; E,
\end{equation}
and we take $\| \nabla \xi \|$ to be the equivalent norm $\| \xi \|_E$ of the space $E$ or the equivalent norm $\| \xi \|_{H_{0}^{1}(\gw)}$. We use $| \, \cdot \, |$ to denote an absolute value or a vector norm in a Euclidean space.

By the Lumer-Phillips theorem and the analytic semigroup generation theorem \cite{SY02}, the linear operator
\begin{equation} \label{opA}
        A =
        \begin{pmatrix}
            d_1 \gd     & 0    &0    &0    &0    &0\\[3pt]
            0 & d_2 \gd    &0    &0    &0   &0\\[3pt]
            0 &0  &d_3 \gd &0    &0   &0\\[3pt]
            0 &0  &0  &d_1\gd  &0    &0\\[3pt]
            0 &0  &0  &0    &d_2 \gd   &0\\[3pt]
            0 &0  &0  &0   &0    &d_3 \gd
        \end{pmatrix}
        : D(A) (= \Pi) \longrightarrow H
\end{equation}
is the generator of an analytic $C_0$-semigroup on the Hilbert space $H$, which will be denoted by $\{e^{At}\}_{t \geq 0}$. By Sobolev embedding theorem, $H_{0}^{1}(\gw) \hookrightarrow L^6(\gw)$ is a continuous embedding for $n \leq 3$. Using the generalized H\"{o}lder inequality, we have
$$
	\| u^{2}v \| \leq \| u \|_{L^6}^2 \| v \|_{L^6}, \quad \| w^{2}z \| \leq \| w \|_{L^6}^2 \| z \|_{L^6}, \quad  \textup{for} \; u, v, w, z \in L^6 (\gw).
$$
By these facts one can verify that the nonlinear mapping
\begin{equation} \label{opF}
    f(g) =
        \begin{pmatrix}
            a - (b+k)u + u^2 v + D_1 (w - u) + N\vp  \\[3pt]
            bu - u^2 v + D_2 (z - v) \\[3pt]
            ku - (\gl + N)\vp + D_3 (\psi - \vp) \\[3pt]
            a - (b+k)w + w^2 z + D_1 (u - w) + N\psi  \\[3pt]
            bw - w^2 z + D_2 (v - z) \\[3pt]
            kw - (\gl + N)\psi + D_3 (\vp - \psi) 
        \end{pmatrix}
        : E \longrightarrow H,
\end{equation}
where $g = (u, v, \vp, w, z, \psi)$, is a locally Lipschitz continuous mapping defined on $E$. Then the initial-boundary value problem \eqref{equ}--\eqref{ic} of this extended Brusselator system is formulated into an initial value problem:
\begin{equation} \label{eveq} 
	\begin{split}
   	&\frac{dg}{dt} = A g + f(g), \quad t > 0. \\
	g(0) = g_0 &= \textup{col} \, (u_0, v_0, \vp_0, w_0, z_0, \psi_0) \in H.
	\end{split}
\end{equation}
Here $g (t) = \textup{col} \, (u(t), v(t), \vp(t), w(t), z(t), \psi(t))$ can be written as $(u(t), v(t), \vp(t), w(t), z(t), \psi(t))$ for simplicity. Accordingly we shall write $g_0 = (u_0, v_0, \vp_0, w_0, z_0, \psi_0)$. 

The following proposition will be used in proving the existence of a weak solution of the initial value problem \eqref{eveq}. Its proof is seen in \cite[Theorem II.1.4]{CV02} and in \cite[Proposition I.3.3]{BP86}.
\begin{proposition} \label{P1} 
Consider the Banach space
\begin{equation} \label{wsp}
	W(0,\tau) = \left\{\eta (\cdot): \eta \in L^2 (0, \tau; E) \; \textup{and}\; \partial_t \eta \in L^2 (0, \tau; E^*)\right\}
\end{equation}
with the norm
$$
	\|\eta \|_W = \|\eta \|_{L^2 (0, \tau; E)} + \|\partial_t \eta \|_{L^2 (0, \tau; E^*)}.
$$
Then the following statements hold:

\textup{(a)} The embedding $W(0, \tau) \hookrightarrow L^2 (0, \tau; H)$ is compact.

\textup{(b)} If $\eta \in W(0, \tau)$, then it coincides with a function in $C([0, \tau]; H)$ for a.e. $t \in [0, \tau]$.

\textup{(c)} If $\eta, \zeta \in W(0, \tau)$, then the function $t \to \inpt{\eta (t), \zeta (t)}_H$ is absolutely continuous on $[0, \tau]$ and
$$
	\frac{d}{dt} \inpt{\eta (t), \zeta (t)} = \left(\frac{d\eta}{dt}, \zeta(t)\right) + \left(\eta(t), \frac{d\zeta}{dt}\right), \; \tup{(} a.e.\tup{)} \; t \in [0, \tau],
$$
where $(\cdot , \cdot)$ is the $(E^*, E)$ dual product.
\end{proposition}
By conducting \emph{a priori} estimates on the Galerkin approximations of the initial value problem \eqref{eveq} and the weak/weak* convergence, we shall be able to prove the local and then global existence and uniqueness of the weak solution $g(t)$ of \eqref{eveq} in the next section, also the continuous dependence of the solutions on the initial data and the regularity properties satisfied by the weak solution. Therefore, the weak solutions for all initial data in $H$ form a semiflow in the space $H$.

We refer to \cite{CV02, SY02, rT88} and many references therein for the concepts and basic facts in the theory of infinite dimensional dynamical systems, including few provided here for clarity. 

\begin{definition} \label{D:abs}
Let $\{S(t)\}_{t \geq 0}$ be a semiflow on a Banach space $X$. A bounded subset $B_0$ of $X$ is called an \emph{absorbing set in} $X$ if, for any bounded subset $B \subset X$, there is some finite time $t_0 \geq 0$ depending on $B$ such that $S(t)B \subset B_0$ for all $t > t_0$.
\end{definition}

\begin{definition} \label{D:asp}
A semiflow $\{S(t)\}_{t \geq 0}$ on a Banach space $X$ is called \emph{asymptotically compact in} $X$ if for any bounded sequences $\{x_n \}$ in $X$ and $\{t_n \} \subset (0, \infty)$ with $t_n \to \infty$, there exist subsequences $\{x_{n_k}\}$ of $\{x_n \}$ and $\{t_{n_k}\}$ of $\{t_n\}$, such that $\lim_{k \to \infty} S(t_{n_k})x_{n_k}$ exists in $X$.
\end{definition}

\begin{definition} \label{D:atr}
Let $\{S(t)\}_{t \geq 0}$ be a semiflow on a Banach space $X$. A subset $\ms{A}$ of $X$ is called a \emph{global attractor in} $X$ for this semiflow, if the following conditions are satisfied: 

(i) $\ms{A}$ is a nonempty, compact, and invariant set in the sense that 
$$
	S(t)\ms{A} = \ms{A} \quad \textup{for any} \; \;  t \geq 0. 
$$

(ii) $\ms{A}$ attracts any bounded set $B$ of $X$ in terms of the Hausdorff distance, i.e.
$$
	\text{dist} (S(t)B, \ms{A}) = \sup_{x \in B} \inf_{y \in \ms{A}} \| S(t) x - y\|_{X} \to 0, \quad \text{as} \; \, t \to \infty.
$$
\end{definition}

Here is the main result of this paper. We \emph{emphasize} that this result is established unconditionally, neither assuming initial data or solutions are nonnegative, nor imposing any restriction on any positive parameters involved in the equations \eqref{equ}--\eqref{eqs}.

\begin{theorem}[Main Theorem] \label{Mthm}
For any positive parameters $d_1, d_2, d_3, D_1, D_2, D_3, a, b, k, \gl $, and $N$, there exists a global attractor $\ms{A}$ in the phase space $H$ for the semiflow $\csg$ of the weak solutions generated by the extended Brusselator evolutionary equation \eqref{eveq}.
\end{theorem}

The following proposition states concisely the basic result on the existence of a global attractor for a semiflow, cf. \cite{jR01, rT88, SY02}.

\begin{proposition} \label{P2}
Let $\csg$ be a semiflow on a Banach space $X$. If the following conditions are satisfied\textup{:} 

\textup{(i)} $\csg$ has an absorbing set in $X$, and 

\textup{(ii)} $\csg$ is  asymptotically compact in $X$, \\
then there exists a global attractor $\ms{A}$ in $X$ for this semiflow, which is given by
$$
    \ms{A} = \om (B_0) \overset{\textup{def}}{=} \bigcap_{\tau \geq 0} \text{Cl}_{X}  \bigcup_{t \geq \tau} (S(t)B_0).
$$
\end{proposition}
In Section 2 we shall prove the global existence of the weak solutions of the extended Brusselator evolutionary equation \eqref{eveq}. In Section 3 we show the absorbing property of this extended Brusselator semiflow. In Section 4 we shall prove the asymptotic compactness of this solution semiflow. In Section 5 we show the existence of a global attractor for this semiflow and its properties as being the $(H, E)$ global attractor and the $\mathbb{L}^\infty$ regularity. We also prove that the global attractor has finite Hausdorff dimension and fractal dimension. 

\section{\textbf{Global Existence of Weak Solutions}}

We shall write $u(t, x), v(t, x),\vp(t, x), w(t, x), z(t, x)$ and $\psi(t,x)$ simply as $u(t), v(t), \vp(t), w(t), z(t)$ and $\psi(t)$, or even as $u, v, \vp, w, z$ and $\psi$. Similarly for other functions of $(t,x)$. 

The local existence of the solution to a system of multicomponent reaction-diffusion equations, such as the initial value problem (IVP) of \eqref{eveq}, with certain regularity requirement is not a trivial issue. There are two different approaches to get a solution. One approach is the mild solutions provided by the "variation of constant formula" in terms of the associated linear semigroup $\{e^{At}\}_{t\geq 0}$ but the parabolic theory of mild solutions requires that $g_0 \in E$ instead of $g_0 \in H$ assumed here. The other approach is the weak solutions obtained through the Galerkin approximations and the Lions-Magenes type of compactness treatment, cf. \cite{CV02, jL69, jR01}. 

\begin{definition}
A function $g(t,x), (t, x) \in [0, \tau] \times \gw$, is called a weak solution to the IVP of the parabolic evolutionary equation \eqref{eveq}, if the following two conditions are satisfied:

(i) $\frac{d}{dt} (g, \eta) = (Ag, \eta) + (f(g), \eta)$ is satisfied for a.e. $t \in [0, \tau]$ and  any $\eta \in E$; 

(ii) $g(t, \cdot) \in L^2 (0, \tau; E) \cap C_w ([0, \tau]; H)$ such that $g(0) = g_0$.

\noindent
Here $(\cdot, \cdot)$ stands for the $(E^*, E)$ dual product.
\end{definition}

For reaction-diffusion systems with more general polynomial nonlinearity of higher degrees, the definition of corresponding weak solutions is made in \cite[Definition XV.3.1]{CV02}.

\begin{lemma} \label{L:locs}
For any given initial data $g_0 \in H$, there exists a unique local weak solution $g(t) = (u(t), v(t), \vp (t), w(t), z(t), \psi (t)), \, t \in [0, \tau]$ for some $\tau > 0$, of the IVP of the extended Brusselator evolutionary equation \eqref{eveq}, which satisfies 
\begin{equation} \label{soln}
    g \in C([0, \tau]; H) \cap C^1 ((0, \tau); H) \cap L^2 (0, \tau; E).
\end{equation}
\end{lemma}

\begin{proof}
Using the orthonormal basis of eigenfunctions $\{e_j (x)\}_{j=1}^\infty$ of the Laplace operator with the homogeneous Dirichlet boundary condition:
$$
	\gd e_j + \gl_j e_j = 0 \;\; \textup{in} \;\, \gw, \quad e_j\mid_{\partial \gw} = 0,  \quad j = 1, 2, \cdots,
$$
we consider the solution 
\begin{equation} \label{gq}
	g_m (t, x) = \sum_{j=1}^m q_j^m (t) e_j (x), \quad t \in [0, \tau], \; x \in \gw,
\end{equation}
of the approximate system
\begin{equation} \label{eveqm} 
	\begin{split}
    \frac{\partial g_m}{\partial t} = A &g_m + P_m f(g_m), \quad t > 0, \\[3pt]
      g_m(0) &= P_m \, g_0 \in H_m,
      \end{split}
\end{equation}
where each $q_j^m (t)$ for $j = 1, \cdots, m$ is a 6-dimensional vector function of $t$ corresponding to the six unknowns $u, v, \vp, w, z$ and $\psi$, and $P_m: H \to H_m = \textup{Span} \{e_1, \cdots, e_m\}$ is the orthogonal projection. Note that for each given integer $m \geq 1$, \eqref{eveqm} can be written as an IVP of a system of ODEs, whose unknown is a $6m$-dimensional vector function of $t$. The components of the solution to this system of ODE are the coefficient functions in the expansion \eqref{gq} of $g_m (t,x)$, namely,
$$
	q^m (t) = \textup{col} \, (q_{ju}^m (t), q_{jv}^m (t), q_{j\vp}^m (t), q_{jw}^m (t), q_{jz}^m (t), q_{j\psi}^m (t); \, j = 1, \cdots, m). 
$$
The IVP of this system of ODE can be written as
\begin{equation} \label{odeg}
	\begin{split}
	\frac{dq^m}{dt} &= \gL_m q^m (t) + F_m (q^m (t)), \; t > 0, \\[3pt]
	q^m (0) = \textup{col} (P_m u_{0j}, \, P_m &v_{0j}, \, P_m \vp_{0j}, \, P_m w_{0j}, \, P_m z_{0j}, \, P_m \psi_{0j}; \, j = 1, \cdots, m).
	\end{split}
\end{equation}
Note that $\gL_m$ is a matrix and $F_m$ is a $6m$-dimensional vector of cubic-degree polynomials of $6m$-variables, which is a locally Lipschitz continuous vector function in $\mathbb{R}^{6m}$. Thus the solution of the system of ODE \eqref{odeg} exists uniquely on a time interval $[0, \tau]$, for some $\tau > 0$. Substituting all the components of this solution $q^m (t)$ into \eqref{gq}, we obtain the unique local solution $g_m (t, x)$ of the IVP \eqref{eveqm}, $m \geq 1$.

By the multiplier method we can conduct \emph{a priori} estimates based on 
\begin{equation*} 
	\frac{1}{2} \| g_m (t) \|_{H_m}^2 + \inpt{\boldsymbol{d} \nabla g_m (t), \nabla g_m (t)}_{H_m} = \inpt{P_m f(g_m (t)), g_m (t)}_{H_m}, \quad t \in [0, \tau],
\end{equation*}
where the matrix $\boldsymbol{d} = \textup{diag} (d_1, d_2, d_3, d_1, d_2, d_3)$. These estimates are similar to what we shall present in the next lemma. Note that $\|g_m (0)\| = \|P_m g_0 \| \leq \|g_0 \|$ for all $m \geq 1$.  It follows that (see the proof of the next lemma)
$$
	\{g_m\}_{m=1}^\infty \, \textup{is a bounded sequence in} \, L^2 (0, \tau; E) \cap L^\infty (0, \tau; H),
$$
so that 
$$
	\{Ag_m\}_{m=1}^\infty \, \textup{is a bounded sequence in} \, L^2 (0, \tau; E^*), \; \; \textup{where} \, E^* = [H^{-1} (\gw)]^6. 
$$
Since $f: E \to H$ is continuous, 
\begin{equation} \label{bdf}
	\{P_m f(g_m)\}_{m=1}^\infty \, \textup{is a bounded sequence in} \, L^2 (0, \tau; H) \subset L^2 (0, \tau; E^*).
\end{equation}
Therefore, by taking subsequences (which we will always relabel as the same as the original sequence), there exist limit functions 
\begin{equation} \label{limgf}
	g(t, \cdot) \in L^2(0, \tau; E) \cap L^\infty (0, \tau; H) \quad \textup{and} \quad \Phi (t, \cdot) \in L^2 (0, \tau; H)
\end{equation}
such that 
\begin{equation} \label{limg}
	\begin{split}
	&g_m \longrightarrow g \; \textup{weakly in} \; L^2 (0, \tau; E),  \, \textup{as} \; m \to \infty, \\
	&g_m \longrightarrow g \; \textup{weak* in} \; L^\infty (0, \tau; H),  \, \textup{as} \; m \to \infty, \\
	&Ag_m \longrightarrow Ag \; \textup{weakly in} \; L^2 (0, \tau; E^*),  \, \textup{as} \; m \to \infty,
	\end{split}
\end{equation}
and
\begin{equation} \label{limf}
	P_m f(g_m) \longrightarrow \Phi \; \textup{weakly in} \; L^2 (0, \tau; H), \, \textup{as} \; m \to \infty.
\end{equation}
To estimate the distributional time derivative sequence $\{\partial_t g_m\}_{m=1}^\infty$, we take the $(E^*, E)$ dual product of the equation \eqref{eveqm} with any $\eta \in E$ to obtain
$$
	\|\partial_t g_m (t) \|_{E^*} \leq \|A g_m (t)\|_{E^*} + c \|P_m f(g_m (t))\|_H, \quad t \in [0, \tau], \, m \geq 1,
$$
where $c$ is a universal constant, so that by the aforementioned boundedness it holds that
$$
	\{\partial_t g_m\}_{m=1}^\infty \, \textup{is a bounded sequence in} \, L^2 (0, \tau; E^*).
$$
From extracting a convergent subsequence and by the uniqueness of distributional derivative it follows that 
\begin{equation} \label{limt}
	\partial_t g_m \longrightarrow \partial_t g \; \textup{weakly in} \; L^2 (0, \tau; E^*),  \, \textup{as} \; m \to \infty.
\end{equation}

In order to show that the limit function $g$ is a weak solution to the IVP of \eqref{eveq}, we need to show $\Phi = f(g)$. By item (a) of Proposition \ref{P1}, the boundedness of $\{g_m\}$ in $L^2 (0, \tau; E)$ and $\{\partial_t g_m\}$ in $L^2 (0, \tau; E^*)$ implies that (in the sense of subsequence extraction)
\begin{equation} \label{strh}
	g_m  \longrightarrow g \; \textup{strongly in} \; L^2 (0, \tau; H),  \, \textup{as} \; m \to \infty.
\end{equation}
Consequently, there exists a subsequence such that
\begin{equation} \label{aeg}
	g_m (t, x) \longrightarrow g(t, x) \; \textup{for a.e.} \; (t, x) \in [0, \tau] \times \gw,  \, \textup{as} \; m \to \infty.
\end{equation}
Due to the continuity we have
\begin{equation} \label{aef}
	P_m f(g_m (t, x))  \longrightarrow f(g(t, x)) \; \textup{for a.e.} \; (t, x) \in [0, \tau] \times \gw,  \, \textup{as} \; m \to \infty.
\end{equation}
According to \cite[Lemma I.1.3]{jL69} or \cite[Lemma II.1.2]{CV02}, the two facts \eqref{bdf} and \eqref{aef} guarantee that 
\begin{equation} \label{limfg}
	P_m f(g_m) \longrightarrow f(g) \; \textup{weakly in} \; L^2 (0, \tau; H), \, \textup{as} \; m \to \infty.
\end{equation}
By the uniqueness, \eqref{limf} and \eqref{limfg} imply that $\Phi = f(g)$ in $L^2 (0, \tau; H)$.

With \eqref{limg}, \eqref{limt} and \eqref{limfg}, by taking limit of the integral of the weak version of \eqref{eveqm} with any given $\eta \in L^2 (0, \tau; E)$, 
$$
	\int_0^\tau \left(\frac{\partial g_m}{\partial t}, \eta\right) \, dt = \int_0^\tau \left[(Ag_m, \eta) + (P_m f(g_m), \eta)\right] dt, \;\; \textup{as} \; m \to \infty,
$$
we obtain
\begin{equation} \label{iw}
	\int_0^\tau \left(\frac{\partial g}{\partial t}, \eta\right) \, dt = \int_0^\tau \left[(Ag, \eta) + (f(g), \eta)\right] dt, \;\; \textup{for any} \; \eta \in L^2 (0, \tau; E).
\end{equation}
Let $\eta \in E$ be any constant function. Since in \eqref{iw} one can replace $[0, \tau]$ by any small time interval, it follows that
\begin{equation} \label{weq} 
	\frac{d}{dt} (g, \eta) = (Ag, \eta) + (f(g), \eta), \;  \textup{for a.e}. \; t \in [0, \tau], \;  \textup{and  for any} \;  \eta \in E.
\end{equation}

Next, $\partial_t g \in L^2 (0, \tau; E^*)$ shown in \eqref{limt} implies that $g \in C_w ([0, \tau]; E^*)$. Since the embedding $H \hookrightarrow E^*$ is continuous, this $g \in C_w ([0, \tau]; E^*)$ and $g \in L^\infty (0, \tau; H)$ shown in \eqref{limg} imply that 
$$
	g \in C_w ([0, \tau]; H),
$$
due to \cite[Theorem II.1.7 and Remark II.1.2]{CV02}. Now we show that $g(0) = g_0$. By item (c) of Proposition \ref{P1}, for any $\eta \in C^1([0, \tau]; E)$ with $\eta (\tau) = 0$, we have
$$
	\int_0^\tau \left(- g, \partial_t \eta\right) \, dt = \int_0^\tau \left[(Ag, \eta) + (f(g), \eta)\right] dt + \inpt{g(0), \eta(0)},
$$
and
$$
	\int_0^\tau \left(- g_m, \partial_t \eta\right) \, dt = \int_0^\tau \left[(Ag_m, \eta) + (P_m f(g_m), \eta)\right] dt + \inpt{P_m g_0, \eta(0)}.
$$
Take the limit of the last equality as $m \to \infty$. Since $P_m g_0 \to g_0$ in $H$, we obtain $\inpt{g(0), \eta(0)} = \inpt{g_0, \eta(0)}$ for any $\eta(0) \in E$. Finally, by the denseness of $E$ in $H$, it implies that $g(0) = g_0$ in $H$. Then by checking against Definition 4, we conclude that the limit function $g$ is a weak solution to the IVP of \eqref{eveq}.

The uniqueness of weak solution can be shown by estimating the difference of any two possible weak solutions with the same initial value $g_0$ through the variation of constant formula and the Gronwall inequality. 

By item (b) of Proposition \ref{P1}, and the fact that the weak solution $g \in W(0, \tau)$, the space defined in \eqref{wsp}, we see that $g \in C ([0, \tau]; H)$, which also infers the continuous dependence of the weak solution on $g_0$. 

Moreover, since $g \in L^2 (0, \tau; E)$, for any $t \in (0, \tau)$ there exists an earlier time $t_0 \in (0, t)$ such that $g (t_0) \in E$. Then the weak solution coincides with the strong solution expressed by the mild solution on $[t_0, \tau]$, cf. \cite{CV02, SY02}, which turns out to be continuously differentiable in time with respect to $H$-norm for $t \in (t_0, \tau)$, cf. \cite[Theorem 48.5]{SY02}. Thus we have shown $g \in C^1 ((0, \tau); H)$. Therefore, the weak solution $g$ satisfies all the properties in \eqref{soln}. The proof is completed.
\end{proof}

\begin{lemma} \label{L:glsn}
For any initial data $g_0 =(u_0, v_0, \vp_0, w_0, z_0, \psi_0) \in H$, there exists a unique global weak solution $g(t) = (u(t), v(t), \vp (t), w(t), z(t), \psi (t)), \, t \in [0, \infty)$, of the IVP of the extended Brusselator evolutionary equation \eqref{eveq} and it becomes a strong solution on the time interval $(0, \infty)$.
\end{lemma}

\begin{proof}
Taking the inner products $\inpt{\eqref{eqv}, v(t)}$ and $\inpt{\eqref{eqz}, z(t)}$ and then summing them up, we get
\begin{equation} \label{vziq}
	\begin{split}
    &\frac{1}{2} \left(\frac{d}{dt} \| v \|^2 + \frac{d}{dt} \| z \|^2 \right) + d_2 \left(\| \nabla v \|^2 + \| \nabla z \|^2\right) \\
     = &\int_{\gw} \left(-u^2 v^2 + buv - w^2 z^2 + bwz - D_2 [v^2 - 2vz + z^2] \right) dx   \\
    = &\int_{\gw} - \left((uv - b/2)^2 + (wz - b/2)^2 + D_2 (v - z)^2 \right)dx +  \frac{1}{2} b^2 |\gw |\leq \, \frac{1}{2} b^2 |\gw |. 
  \end{split}
\end{equation}
By Poincar\'{e} inequality, it follows that
\begin{equation*}
    \frac{d}{dt} \left(\| v \|^2 + \| z \|^2 \right) + 2\ga d_2 \left(\| v \|^2 + \| z \|^2 \right) \leq b^2 |\gw |,
\end{equation*}
which yields
\begin{equation} \label{vz}
    \| v(t) \|^2 + \| z (t)\|^2 \leq e^{- 2\ga d_2 t} \left(\| v_0 \|^2 + \| z_0 \|^2 \right) + \frac{b^2 |\gw |}{2\ga d_2}, \quad \text{for} \;\,  t \in [0, T_{max}).
\end{equation}

In order to treat the $(u, \vp)$-components and the $(w, \psi)$-components, we shall go through several steps. Let $y(t, x) = u(t, x) + v(t, x) + w(t,x) + z(t,x)$ and $\xi (t, x) = \vp(t, x) + \psi (t, x)$. Add up the equations \eqref{equ}, \eqref{eqv}, \eqref{eqw} and \eqref{eqz}  and add up \eqref{eqp} and \eqref{eqs}, respectively, to get the following equations satisfied by $y(t, x)$ and $\xi (t, x)$, 
\begin{align*} 
    \frac{\pdr y}{\pdr t} &= d_1 \gd y - ky + \left[(d_2 - d_1)\gd (v + z) + k(v + z) + 2a\right] + N\xi,  \\ 
    \frac{\pdr \xi}{\pdr t} &= d_3 \gd \xi + ky - k(v + z) - (\gl + N)\xi. 
\end{align*}
Rescaling $\xi = \mu \varXi$ with $\mu = k/N$, we get 
\begin{align} 
    \frac{\pdr y}{\pdr t} &= d_1 \gd y - ky + \left[(d_2 - d_1)\gd (v + z) + k(v + z) + 2a\right] + k\varXi, \label{eqy} \\ 
    \mu \frac{\pdr \varXi}{\pdr t} &= \mu d_3 \gd \varXi + ky - k(v + z) - (\mu \gl + k)\varXi. \label{eqx}
\end{align}

Now taking the inner-product $\inpt{\eqref{eqy},y(t)}$ we obtain
\begin{align*}
    \frac{1}{2}& \frac{d}{dt} \| y \|^2 + d_1 \| \nabla y \|^2 \leq \int_{\gw} \, \left[(d_2 - d_1)\gd (v+z) + (v+ z) + 2a\right]y \, dx - k \| y \|^2  + k\|y\| \|\varXi \| \\[3pt]
     \leq&\, | d_1 - d_2 | \| \nabla (v + z) \| \| \nabla y \| + \| v + z \| \| y \| + 2a | \gw |^{1/2} \| y \| - k \| y \|^2  + k\|y\| \|\varXi\| \\[3pt]
     \leq &\, \frac{d_1}{4} \| \nabla y \|^2 + \frac{|d_1 - d_2 |^2}{d_1} \| \nabla (v + z) \|^2 + \frac{d_1 \ga}{4} \| y \|^2 + \frac{2}{d_1 \ga}\| v + z \|^2 +  \frac{8}{d_1 \ga}a^2 | \gw | \\
    &- k \| y \|^2  + k\|y\| \|\varXi \| \\
     \leq &\, \frac{d_1}{2} \| \nabla y \|^2 + \frac{|d_1 - d_2 |^2}{d_1} \| \nabla (v + z) \|^2 + \frac{2}{d_1 \ga}\| v + z \|^2 +  \frac{8}{d_1 \ga}a^2 | \gw | - k \| y \|^2  + k\|y\| \|\varXi \|,
\end{align*}
so that
\begin{equation} \label{nyiq}
	\begin{split}
    \frac{d}{dt} \| y \|^2 + d_1 \| \nabla y \|^2 \leq &\, \frac{2 | d_1 - d_2 |^2}{d_1} \| \nabla (v+z) \|^2 + \frac{8}{d_1 \ga} \left(\| v\|^2 + \|z \|^2 + 2a^2 | \gw | \right) \\
    & - 2k \| y \|^2  + 2k\|y\| \|\varXi\|.
    	\end{split}
\end{equation}
Then by substituting \eqref{vz} into \eqref{nyiq} we get
\begin{equation} \label{yiq}
    \frac{d}{dt} \| y \|^2 + d_1 \| \nabla y \|^2 \leq \frac{2 | d_1 - d_2 |^2}{d_1} \| \nabla (v+z) \|^2  - 2k \| y \|^2  + 2k\|y\| \|\varXi\| + C_1(v_0, z_0, t),
\end{equation}
where
\begin{equation} \label{c1}
    C_1(v_0, z_0, t) = \frac{8}{d_1 \ga} e^{- 2\ga d_2 t} \left(\| v_0 \|^2 + \| z_0 \|^2\right) + \left(\frac{4b^2}{\ga^2 d_1 d_2} + \frac{16 a^2}{d_1 \ga} \right) |\gw |.
\end{equation}

Taking the inner-product $\inpt{\eqref{eqx},\varXi(t)}$ we obtain
\begin{align*} 
	\mu \frac{d}{dt} &\|\varXi \|^2 + 2\mu d_3 \|\nabla \varXi \|^2 \leq 2k \|v + z\| \|\varXi\| + 2k\|y\| \|\varXi \| - 2(\mu \gl + k) \| \varXi \|^2 \\
	&\leq \mu d_3 \ga \|\varXi \|^2 + \frac{k^2 (\|v \| + \|z \|)^2}{\mu d_3 \ga} + 2k\|y\| \|\varXi \| - 2(\mu \gl + k) \| \varXi \|^2,
\end{align*}
where $\mu d_3 \ga \|\varXi \|^2 \leq \mu d_3 \|\nabla \varXi \|^2$, so that
\begin{equation} \label{xiq}
	\mu \frac{d}{dt} \| \varXi \|^2 + \mu d_3 \| \nabla \varXi \|^2 \leq 2k\|y\| \|\varXi \| - 2(\mu \gl + k) \| \varXi \|^2 + C_2 (v_0, z_0, t),
\end{equation}
in which, by \eqref{vz}, 
\begin{equation} \label{c2}
	C_2 (v_0, z_0, t) = \frac{2k^2}{\mu d_3 \ga} \left[ e^{- 2\ga d_2 t} \left(\| v_0 \|^2 + \| z_0 \|^2\right) + \frac{b^2 |\gw |}{2 \ga d_2}\right]. 
\end{equation}

Then adding up \eqref{yiq} and \eqref{xiq} and noticing that 
$$
	- 2k \| y \|^2  + 4k\|y\| \|\varXi\| - 2(\mu \gl + k) \| \varXi \|^2 = - 2\left[k(\|y\| - \|\varXi\|)^2 + \mu \gl \|\varXi \|^2\right] \leq 0,
$$
we obtain
\begin{equation} \label{ypx}
	\begin{split}
	 \frac{d}{dt} (\|y\|^2 &+ \|\mu^{-1/2} \xi \|^2) + \min \{d_1, d_3 \} (\|\nabla y \|^2 + \|\mu^{-1/2} \nabla \xi \|^2) \\
	&\leq \frac{2 | d_1 - d_2 |^2}{d_1} \| \nabla (v+z) \|^2 + C_1 + C_2,
	\end{split}
\end{equation}
where $C_1$ and $C_2$ are given by \eqref{c1} and \eqref{c2}. Integration of the inequality \eqref{ypx} with the term $\min \{d_1, d_3 \} (\|\nabla y \|^2 + \|\mu^{-1/2} \nabla \xi \|^2)$ dropped yields the following estimate,
\begin{equation} \label{yx}
	\begin{split}
    \| y(t) \|^2 + \| \xi (t) \|^2 \leq &\, \frac{\max \{1, \mu^{-1}\}}{\min \{1, \mu^{-1}\}} \left(\| u_0 + v_0 + w_0 + z_0\|^2 + \| \vp_0 + \psi_0 \|^2\right) \\
    {}& + \frac{2| d_1 - d_2 |^2}{\min \{1, \mu^{-1} \}d_1} \int_{0}^{t} \| \nabla (v(s)+z(s)) \|^2 \, ds  \\[3pt]
    {} &+ C_3 \left(\| v_0 \|^2 + \| z_0 \|^2\right) + C_4 |\gw | \, t, \quad t \in [0, T_{max}),
   \end{split}
\end{equation}
where we get from \eqref{c1} and \eqref{c2} that 
$$
	C_3 = \frac{\max \{1, \mu\}}{2\ga d_2} \left(\frac{8}{d_1 \ga} +  \frac{2k^2}{\mu d_3 \ga} \right), \quad  C_4 = \max \{1, \mu \} \left(\frac{4b^2}{\ga^2 d_1 d_2} + \frac{16 a^2}{d_1 \ga} + \frac{k^2 b^2}{\mu \ga^2 d_2 d_3} \right).
$$
From \eqref{vziq}, for $t \in [0, T_{max})$ we have
\begin{equation} \label{nvz}
    d_2 \int_0^{t} \| \nabla (v(s) \pm z(s)) \|^2 \, ds \leq 2d_2 \int_0^{t} \left(\|\nabla v(s) \|^2 + \|\nabla z(s) \|^2\right)\, ds \leq \left(\| v_0 \|^2 + \|z_0 \|^2\right) + b^2 | \gw | t. 
\end{equation}
Substitute this into \eqref{yx} to obtain
\begin{equation} \label{yx1}
	\begin{split}
    \| y(t) \|^2 + \| \xi (t)\|^2 \leq &\, \frac{\max \{1, \mu^{-1}\}}{\min \{1, \mu^{-1}\}} \left(\| u_0 + v_0 + w_0 + z_0\|^2 + \| \vp_0 + \psi_0 \|^2\right) \\[3pt]
    {}& +\left( \frac{2| d_1 - d_2 |^2}{\min \{1, \mu^{-1} \}d_1 d_2} + C_3 \right)\left( \| v_0 \|^2 + \|z_0 \|^2\right)\\[3pt]
    {} &+ \left(\frac{2| d_1 - d_2 |^2 b^2}{\min \{1, \mu^{-1} \}d_1 d_2} + C_4 \right) |\gw | \,t, \quad t \in [0, T_{max}). 
  \end{split}
\end{equation}
From \eqref{vz} and \eqref{yx1} it follows that
\begin{equation} \label{pp}
	\begin{split}
   & \|u(t) + w(t) \|^2  + \|\vp (t) + \psi (t) \|^2 = \|y(t) - (v(t) + z(t))\|^2 + \|\vp (t) + \psi (t) \|^2 \\[3pt]
     \leq &\, 3 (\|y(t) \|^2 + \|v(t)\|^2 + \|z(t) \|^2) + \|\xi \|^2 \leq C_5 (\|g_0 \|^2 + |\gw | t), \quad  t \in [0, T_{max}),
   \end{split}
\end{equation}
where $C_5 > 0$ is a constant independent of $g_0$ and $t$.

On the other hand, let $p (t,x) = u(t, x) + v(t, x) - w(t,x) - z(t,x)$ and $\theta (t, x) = \vp (t, x) - \psi (t, x)$, which satisfy the equations
\begin{align*} 
    \frac{\pdr p}{\pdr t} &= d_1 \gd p - (k + 2D_1) p + (d_2 - d_1)\gd (v - z) + (k + 2(D_1 - D_2))(v - z)+ N\theta,  \\
    \frac{\pdr \theta}{\pdr t} &= d_3 \gd \theta + k p - k(v - z) - (\gl + N))\theta - 2D_3 \theta. 
\end{align*}
Rescaling $\theta = \mu \varTheta$ with the same $\mu$ as above, we get
\begin{align} 
    \frac{\pdr p}{\pdr t} &= d_1 \gd p - (k + 2D_1) p + (d_2 - d_1)\gd (v - z) + (k + 2(D_1 - D_2))(v - z)+ k\varTheta,  \label{eqpp} \\
    \mu \frac{\pdr \varTheta}{\pdr t} &= \mu d_3 \gd \varTheta + k p - k(v - z) - (\mu \gl + k))\varTheta - 2\mu D_3 \varTheta. \label{eqzt}
\end{align}

Taking the inner-product $\inpt{\eqref{eqpp}, p(t)}$ we obtain
\begin{equation} \label{npiq}
	\begin{split}
	 &\frac{1}{2} \frac{d}{dt} \|p \|^2 + d_1 \|\nabla p \|^2 + (k + 2D_1) \|p \|^2 \\[5pt]
	 \leq &\, (d_1 - d_2) \|\nabla (v - z)\| \|\nabla p \| + |k + 2(D_1 - D_2)| \|v - z\| \|p\| + k \|\varTheta \| \|p \| \\[3pt]
	 \leq &\, \frac{d_1}{2} \|\nabla p\|^2 + \frac{|d_1 - d_2|^2}{2d_1} \|\nabla (v-z)\|^2 + D_1\|p\|^2 \\[3pt]
	 &+ \frac{1}{4D_1} |k + 2(D_1 - D_2)|^2 \|v - z\|^2 + k \|\varTheta \| \|p \|,
	 \end{split}
\end{equation}
so that 
\begin{equation} \label{piq}
	\begin{split}
	\frac{d}{dt} \|p \|^2 + d_1 \|\nabla p \|^2 + D_1\|p \|^2 &\leq \frac{|d_1 - d_2|^2}{d_1} \|\nabla (v-z)\|^2 + C_6 (v_0, z_0, t) \\[3pt]
	& - 2k \| p \|^2 + 2k \|\varTheta \| \| p \|,
	\end{split}
\end{equation}
where
\begin{equation} \label{c6}
    C_6(v_0, z_0, t) = \frac{1}{D_1} |k + 2(D_1 - D_2)|^2 \left[ e^{- 2\ga d_2 t} \left(\| v_0 \|^2 + \| z_0 \|^2\right) + \frac{b^2|\gw|}{2\ga d_2} \right].
\end{equation}

Taking the inner-product $\inpt{\eqref{eqzt}, \varTheta (t)}$ we get
\begin{align*}
	\mu \frac{d}{dt} \|\varTheta \|^2 + \mu d_3 \|\nabla \varTheta \|^2 & \leq 2k \|v - z\|\|\varTheta \| + 2k \| p \| \|\varTheta \| - 2(\mu \gl + k + 2\mu D_3)\|\varTheta \|^2 \\[3pt]
	& \leq \frac{k^2 (\| v \|^2 + \| z \|^2)}{2\mu D_3} + 2k \| p \| \|\varTheta \| - 2(\mu \gl + k)\|\varTheta \|^2,
\end{align*}
so that 
\begin{equation} \label{thiq}
	\mu \left(\frac{d}{dt} \|\varTheta \|^2 + d_3 \| \nabla \varTheta \|^2 + \gl \| \varTheta \|^2\right) \leq  2k \| p \| \|\varTheta \| - 2k\|\varTheta \|^2 + C_7 (v_0, z_0, t),
\end{equation}
where
\begin{equation} \label{c7}
	C_7 (v_0, z_0, t) = \frac{k^2}{2\mu D_3} \left[e^{-2\ga d_2 t} (\|v_0\|^2 + \|z_0\|^2) +\frac{b^2 |\gw|}{2\ga d_2}\right].
\end{equation}

Summing up \eqref{piq} and \eqref{thiq} and noticing that
$$
	-2k \|p\|^2 + 4k \|\varTheta\| \|p\| - 2k \|\vartheta \|^2 = -2k (\|p\| - \|\varTheta\|)^2 \leq 0,
$$
we have
\begin{equation} \label{pth}
	\begin{split}
	&\frac{d}{dt} (\|p\|^2 + \|\mu^{-1} \theta \|^2) + \min \{d_1, d_3\} (\|\nabla p \|^2 + \|\mu^{-1} \nabla \theta \|^2) \\
	 \leq &\, \frac{|d_1 - d_2|^2}{d_1} \|\nabla (v-z)\|^2 - (D_1 \|p\|^2 + \mu^{-1} \gl \|\theta \|^2) + C_6 + C_7,
	\end{split}
\end{equation}	
where $C_6$ and $C_7$ are given by \eqref{c6} and \eqref{c7}. We drop the term $\min \{d_1, d_3\} (\|\nabla p \|^2 + \|\mu^{-1} \nabla \theta \|^2)$ from the left side and the term $- (D_1 \|p\|^2 + \mu^{-1} \gl \|\theta \|^2)$ from the right side of \eqref{pth}. Then integrate the resulting inequality and use \eqref{nvz} to get the following estimate,
\begin{equation} \label{pth1}
	\begin{split}
	\|p(t)\|^2 &+ \|\theta (t)\|^2  \leq  \frac{\max \{1, \mu^{-1}\}}{\min \{1, \mu^{-1}\}} \left(\|u_0+v_0 - w_0 - z_0\|^2 + \|\vp_0  - \psi_0 \|^2\right) \\[3pt]
	&+ \frac{1}{d_1} \max \{1, \mu\} |d_1 - d_2|^2 \int_0^t \|\nabla (v(s) - z(s) )\|^2\, ds \\[3pt]
         &+ \frac{\max \{1, \mu \}}{2\ga d_2} \left(\frac{k^2}{2\mu D_3} + \frac{1}{D_1} |k + 2(D_1 - D_2)|^2\right) \left(\|v_0\|^2 + \|z_0\|^2 + b^2 |\gw | t \right) \\[3pt]
         &\leq  \frac{\max \{1, \mu^{-1}\}}{\min \{1, \mu^{-1}\}} \left(\|u_0+v_0 - w_0 - z_0\|^2 + \|\vp_0  - \psi_0 \|^2\right) \\[3pt]
         &+ C_8 \left(\|v_0\|^2 + \|z_0\|^2 + b^2 |\gw | t \right), \; t \in [0, T_{max}), 
	\end{split}
\end{equation}
where
$$
	C_8 = \frac{\max \{1, \mu \}|d_1 - d_2|^2}{d_1d_2}  + \frac{\max \{1, \mu \}}{2\ga d_2} \left(\frac{k^2}{2\mu D_3} + \frac{1}{D_1} |k + 2(D_1 - D_2)|^2\right).
$$

From \eqref{vz} and \eqref{pth1} it follows that
\begin{equation} \label{mm}
	\begin{split}
   	& \|u(t) - w(t) \|^2 + \|\vp (t) - \psi (t)\|^2 = \| p(t) - (v(t) - z(t)) \|^2 + \|\vp (t) - \psi (t)\|^2 \\[3pt]
    	\leq &\, 3\left(\|p(t)\|^2 + \|v(t) \|^2 + \|z(t)\|^2\right) + \|\theta (t)\|^2 \leq C_9 (\|g_0\|^2 + |\gw | t) \quad  t \in [0, T_{max}),
   	\end{split}
\end{equation}
where $C_9 > 0$ is a constant independent of initial data $g_0$ and time $t$.

Finally, from \eqref{vz}, \eqref{pp} and \eqref{mm} we conclude that for any initial data $g_0 \in H$, all the six components of the weak solution $g(t) = g(t; g_0)$, i.e. $v(t), z(t)$, and
\begin{equation} \label{4com}
	\begin{split}
	u(t)& = \frac{1}{2} [(u(t) + w(t)) + (u(t) - w(t))], \quad w(t) = \frac{1}{2} [(u(t) + w(t)) - (u(t) - w(t))], \\
	\vp (t) & = \frac{1}{2} [(\vp(t) + \psi (t)) + (\vp(t) - \psi(t))], \quad \psi(t) = \frac{1}{2} [(\vp(t) + \psi(t)) - (\vp(t) - \psi(t))],
	\end{split}
\end{equation}
are uniformly bounded in $[0, T_{max})$ if $T_{max}$ is finite. Therefore, for any $g_0 \in H$, the weak solution $g(t)$ of \eqref{eveq} will never blow up in $H$ at any finite time. It shows that for any $g_0 \in H$, $T_{max} = \infty$, and the weak solution exists globally. Moreover, as shown in the last paragraph of the proof of Lemma \ref{L:locs}, any weak solution turns out to be a strong solution on the time interval $(0, \infty)$. The proof is completed.
\end{proof}

\section{\textbf{Absorbing Properties}}

By the global existence and uniqueness of the weak solutions and their continuous dependence on initial data shown in Lemma \ref{L:locs} and Lemma \ref{L:glsn}, the family of all the global weak solutions $\{g(t; g_0): t \geq 0, g_0 \in H \}$ defines a semiflow on $H$,
$$
	S(t): g_0 \mapsto g(t; g_0),  \quad g_0 \in H, \; t \geq 0,
$$
which will be called the \emph{extended Brusselator semiflow} associated with the extended Brusselator evolutionary equation \eqref{eveq}.

\begin{lemma} \label{L:absb}
There exists a constant $K_1 > 0$, such that the set 
\begin{equation} \label{bk}
    	B_0 = \left\{g \in H : \| g \|^2 \leq K_1 \right\} 
\end{equation}
is an absorbing set in $H$ for the extended Brusselator semiflow $\{S(t)\}_{ t \geq 0}$.
\end{lemma}
\begin{proof}
For the components $v(t)$ and $z(t)$ of this semiflow $\{S(t)\}_{ t \geq 0}$, from \eqref{vz} we obtain
\begin{equation} \label{vzsup}
    	\limsup_{t \to \infty} \, (\| v(t) \|^2 + \|z(t)\|^2) < R_0 = \frac{b^2 |\gw |}{\ga d_2},
\end{equation}
and for any given bounded set $B \subset H$ there exists a finite time $T_B^{v,z} \geq 0$ such that $\|v(t)\|^2 + \|z(t)\|^2 < R_0$ for any $g_0 \in B$ and all $t > T_B^{v,z}$. For any $t \geq 0$, \eqref{vziq} implies that
\begin{equation} \label{vztt}
	\begin{split}
    	\int_{t}^{t+1}& (\| \nabla v(s) \|^2 + \|\nabla z(s)\|^2) \, ds \leq \frac{1}{d_2} (\| v(t) \|^2 + \|z(t)\|^2 + b^2 |\gw |) \\[3pt]
    	&\leq \frac{1}{d_2} \left[ e^{- 2\ga d_2 t} (\| v_0 \|^2 + \|z_0\|^2) + \left(1 + \frac{1}{2\ga d_2}\right)b^2 |\gw |\right], 
	\end{split}
\end{equation}
which is for later use.

Let $d = \min \{d_1, d_3\}$. From \eqref{ypx} and by Poincar\'{e} inequality we can deduce that
\begin{equation} \label{eyxiq}
	\begin{split}
    &\frac{d}{dt} \left( e^{\ga d t} (\| y(t) \|^2 + \|\mu^{-1/2}  \xi (t) \|^2) \right) \\
    \leq &\, \frac{2| d_1 - d_2 |^2}{d_1} \, e^{\ga dt} \| \nabla (v(t) + z(t)) \|^2 + e^{\ga d t} (C_1(v_0, z_0, t) + C_2 (v_0, z_0, t)), \quad t > 0.
    	\end{split}
\end{equation}
Integrate \eqref{eyxiq} to obtain
\begin{equation} \label{eyx}
	\begin{split}
    \| y(t) \|^2 + \| \xi (t) \|^2 \leq &\, e^{-\ga dt}  \frac{\max \{1, \mu^{-1}\}}{\min \{1, \mu^{-1}\}} (\| u_0 + v_0 +w_0+z_0 \|^2 + \|\vp_0 + \psi_0 \|^2) \\[5pt]
    &+\frac{\max \{1, \mu\} 2| d_1 - d_2 |^2}{d_1} \int_{0}^{t} e^{- \ga d(t -\tau)} \| \nabla (v(\tau) + z(\tau))\|^2 \, d\tau  \\[5pt]
    &+ C_{10} (v_0, z_0, t), \quad t > 0,
    	\end{split}
\end{equation}
where by \eqref{c1} and \eqref{c2} we have
\begin{align*} 
    & C_{10} (v_0, z_0, t) = \max \{1, \mu\} \left[\frac{8}{d_1 \ga} e^{-\ga d t}  \int_{0}^{t}  e^{\ga (d_1 - 2 d_2)\tau}\, d\tau \, (\| v_0 \|^2 + \|z_0\|^2)  \right. \\[5pt]
    & \left. + \frac{1}{\ga d} \left(\frac{4b^2}{\ga d_1 d_2} + \frac{16 a^2}{\ga d_1} \right) |\gw | + \frac{2k^2}{\mu \ga d_3}e^{-\ga d t}  \int_{0}^{t}  e^{\ga (d_1 - 2 d_2)\tau}\, d\tau \, (\| v_0 \|^2 + \|z_0\|^2) + \frac{k^2 b^2}{\mu \ga^3 d d_2 d_3} |\gw | \right] \\[5pt]
    &= \max \{1, \mu\} \left[\left(\frac{8}{d_1 \ga} + \frac{2k^2}{\mu \ga d_3} \right) \beta (t) (\| v_0 \|^2 + \|z_0\|^2) + \left(\frac{4b^2}{\ga^3 d d_1 d_2} + \frac{16 a^2}{\ga^2 d d_1} +\frac{k^2 b^2}{\mu \ga^3 d d_2 d_3}\right) |\gw | \right],
\end{align*}
in which
\begin{equation} \label{apht}
    \beta (t) = e^{-\ga dt}  \int_{0}^{t} e^{\ga(d_1 - 2 d_2)\tau}\, d\tau = 
    \begin{cases}
        \frac{1}{|\ga (d_1 - 2d_2) |}\left| e^{- 2\ga d_2 t} - e^{-\ga dt} \right|,  & \text{if  $d_1 \ne 2d_2$,} \\[10pt]
        t e^{-\ga dt} \leq \frac{2}{\ga d} e^{-1} e^{-\ga dt/2}, & \text{if $d_1 = 2d_2$.} 
    \end{cases}
\end{equation}
Note that $\beta (t) \to 0$, as $t \to \infty$. On the other hand, multiplying \eqref{vziq} by $e^{\ga dt}$ and then integrating each term of the resulting inequality, we get
$$
    \frac{1}{2} \int_{0}^{t} \, e^{\ga d\tau} \frac{d}{d\tau} \left(\| v(\tau) \|^2 + \|z(\tau)\|^2 \right) \, d\tau + d_2 \int_{0}^{t} \, e^{\ga d\tau} (\| \nabla v(\tau) \|^2 + \|\nabla z(\tau)\|^2) \, d\tau \leq \frac{b^2}{2\ga d} | \gw | e^{\ga dt}.
$$
By integration by parts and using \eqref{vz}, we find that
\begin{equation} \label{evz}
	\begin{split}
    & d_2 \int_{0}^t \, e^{\ga d\tau} (\| \nabla v(\tau) \|^2 + \| \nabla z(\tau) \|^2 )\, d\tau \leq \frac{b^2}{2\ga d} \, | \gw |\, e^{\ga  d \, t} \\
    & - \frac{1}{2} \left[e^{\ga dt} (\| v(t) \|^2 + \| z(t) \|^2 ) - (\| v_0 \|^2 + \|z_0\|^2) - \int_{0}^{t} \, \ga d \, e^{\ga d\tau} (\| v(\tau) \|^2 + \|z(\tau)\|^2) \, d\tau \right]   \\
     \leq &\, \frac{b^2}{\ga d} |\, \gw | \, e^{\ga d\, t} + (\| v_0 \|^2 + \|z_0\|^2) +  \ga d \int_{0}^{t} e^{\ga(d_1 - 2d_2)\tau} (\| v_0 \|^2 + \|z_0\|^2)\, d\tau + \frac{b^2 |\gw |}{2\ga d_2} e^{\ga d\, t}   \\
     \leq &\, \frac{1}{\ga} \left(\frac{1}{d} + \frac{1}{d_2} \right) b^2 |\gw | e^{\ga d\, t} + \left(1 + \ga d \, \beta (t)  e^{\ga d\, t} \right) (\| v_0 \|^2 + \|z_0\|^2), \quad  t \geq 0. 
    	\end{split}
\end{equation}
Substituting \eqref{evz} into \eqref{eyx} we obtain that, for $t \geq 0$,

\begin{equation} \label{yxf}
	\begin{split}
    \| &y(t) \|^2 + \|\xi (t)\|^2  \leq e^{-\ga dt}  \frac{\max \{1, \mu^{-1}\}}{\min \{1, \mu^{-1}\}} (\| u_0 + v_0 +w_0+z_0\|^2 + \| \vp_0 + \psi_0 \|^2)+ C_{10} (v_0, z_0, t)   \\[5pt]
     &+ \frac{2\max \{1, \mu\}| d_1 - d_2 |^2}{d_1\, d_2} e^{-\ga dt} \left[ \frac{1}{\ga}\left(\frac{1}{d} + \frac{1}{d_2} \right) b^2 |\gw | e^{\ga dt} + \left(1 + \ga d \, \beta (t) e^{\ga dt}\right) (\| v_0 \|^2+ \|z_0\|^2)\right]   \\[3pt]
     &  = e^{-\ga dt}  \frac{\max \{1, \mu^{-1}\}}{\min \{1, \mu^{-1}\}} (\| u_0 + v_0 +w_0+z_0\|^2 + \| \vp_0 + \psi_0 \|^2)   \\[3pt]
     & +\max \{1, \mu\} \left[\left(\frac{8}{d_1 \ga} + \frac{2k^2}{\mu \ga d_3} \right) \beta (t) (\| v_0 \|^2 + \|z_0\|^2) + \left(\frac{4b^2}{\ga^3 d d_1 d_2} + \frac{16 a^2}{\ga^2 d d_1} +\frac{k^2 b^2}{\mu \ga^3 d d_2 d_3}\right) |\gw | \right] \\[3pt]
     &+ \frac{2\max \{1, \mu\}| d_1 - d_2 |^2}{d_1\, d_2} \left[ \frac{1}{\ga}\left(\frac{1}{d} + \frac{1}{d_2} \right) b^2 |\gw | + \left(e^{-\ga dt} + \ga d \, \beta (t)\right) (\| v_0 \|^2+ \|z_0\|^2)\right].
     	\end{split}
\end{equation}
Since $\beta (t) \rightarrow 0$ and $e^{- \ga dt} \rightarrow 0$ as $t \to \infty$, from \eqref{yxf} we come up with
\begin{equation} \label{yxsup}
    \limsup_{t \to \infty} \, (\|y(t) \|^2 + \|\xi (t)\|^2) < R_1,
\end{equation}
where
\begin{equation} \label{r1}
	R_1 = 1 +\max \{1, \mu\} \left(\frac{4b^2}{\ga^3 d d_1 d_2} + \frac{16 a^2}{\ga^2 d d_1} +\frac{k^2 b^2}{\mu \ga^3 d d_2 d_3} + \frac{2| d_1 - d_2 |^2}{\ga d_1\, d_2}\left(\frac{1}{d} + \frac{1}{d_2} \right) b^2 \right) |\gw |.
\end{equation}
Moreover, for any given bounded set $B \subset H$ there exists a finite time $T_B^{y,\xi} \geq 0$ such that $\|y(t)\|^2 + \|\xi (t)\|^2 < R_1$ for any $g_0 \in B$ and all $t > T_B^{y,\xi}$. The combination of \eqref{vzsup} and \eqref{yxsup} gives rise to
\begin{equation} \label{psup}
	\begin{split}
    	\limsup_{t \to \infty} \,& (\| u(t) + w(t) \|^2 + \| \vp (t) + \psi (t)\|^2) \\
	& =  \limsup_{t \to \infty} \, \| y(t) - (v(t) + z(t))\|^2 + \| \vp (t) + \psi (t)\|^2) \\ 
	& \leq \limsup_{t \to \infty} \left(3 (\|y(t)\|^2 + \|v(t)\|^2 + \|z(t)\|^2) + \|\xi \|^2 \right) < 3(R_0 + R_1).
    	\end{split}
\end{equation}

From the inequality \eqref{pth} satisfied by $p (t) = u(t) + v(t) - w(t) -z(t)$ and $\theta (t) = \vp (t) - \psi (t)$ and by Poincar\'{e} inequality, similarly we get
\begin{equation} \label{epthiq}
	\begin{split}
	\frac{d}{dt} \left( e^{\ga dt} (\|p(t) \|^2 + \|\mu^{-1} \theta(t) \|^2\right) \leq &\, \frac{| d_1 - d_2 |^2}{d_1} \, e^{\ga dt} \| \nabla (v(t) - z(t)) \|^2 \\
	& + e^{\ga dt} (C_6(v_0, z_0, t) + C_7(v_0, z_0, t)), \quad t > 0.
		\end{split}
\end{equation}
Integrate \eqref{epthiq} to obtain
\begin{equation} \label{epth}
	\begin{split}
    &\| p(t) \|^2 + \|\theta (t)\|^2 \leq e^{-\ga dt}  \frac{\max \{1, \mu^{-1}\}}{\min \{1, \mu^{-1}\}} (\| u_0 + v_0 - w_0 - z_0 \|^2 + \|\vp_0 - \psi_0 \|^2) \\
    &+ \max \{1, \mu\} \left(\frac{| d_1 - d_2 |^2}{d_1} \, \int_{0}^{t} e^{- \ga d(t -\tau)} \| \nabla (v(\tau) - z(\tau))\|^2 \, d\tau + C_{11} (v_0, z_0, t)\right),
    	\end{split}
\end{equation}
where
\begin{equation} \label{c11}
    C_{11} (v_0, z_0, t) = \left(\frac{1}{D_1} |k + 2(D_1 - D_2)|^2 + \frac{k^2}{2\mu D_3}\right)\left(  \beta (t) (\| v_0 \|^2 + \|z_0\|^2)+ \frac{b^2}{2\ga^2 d d_2} |\gw |\right). \\
\end{equation}
Using \eqref{evz} to treat the integral term in \eqref{epth}, we obtain 
\begin{equation} \label{pthif}
	\begin{split}
   & \| p(t) \|^2 + \|\theta (t)\|^2 \leq e^{-\ga dt} \frac{\max \{1, \mu^{-1}\}}{\min \{1, \mu^{-1}\}} (\| u_0 + v_0 - w_0 - z_0\|^2 + \|\vp_0 - \psi_0 \|^2) \\
     &+ \max \{1, \mu\} \left(\frac{2| d_1 - d_2 |^2}{d_1} \int_0^t e^{- \ga d(t -\tau)} (\| \nabla v(\tau) \|^2 + \|\nabla z(\tau)\|^2) \, d\tau  + C_{11} (v_0, z_0, t) \right)\\
     & \leq e^{-\ga dt} \frac{\max \{1, \mu^{-1}\}}{\min \{1, \mu^{-1}\}} (\| u_0 + v_0 - w_0 - z_0\|^2 + \|\vp_0 - \psi_0 \|^2)  + \max \{1, \mu \} C_{11} (v_0, z_0, t) \\
     & + \max \{1, \mu\} \frac{2| d_1 - d_2 |^2}{d_1 d_2}\left[ \frac{1}{\ga} \left(\frac{1}{d} + \frac{1}{d_2} \right) b^2 |\gw |  + \left(e^{-\ga dt} + \ga d \, \beta (t)\right) (\| v_0 \|^2 + \|z_0\|^2) \right].
         	\end{split}
\end{equation}
Again, since $\beta (t) \rightarrow 0$ and $e^{- \ga dt} \rightarrow 0$ as $t \to 0$, from \eqref{c11} and \eqref{pthif} we get
\begin{equation} \label{pthsup}
    \limsup_{t \to \infty} \, (\| p(t) \|^2 + \|\theta (t)\|^2) < R_2,
\end{equation}
where
\begin{equation} \label{r2}
	R_2 = 1 + \max \{1, \mu\} \left[\frac{2|d_1 - d_2|^2}{\ga d_1 d_2} \left(\frac{1}{d} + \frac{1}{d_2}\right) + \frac{1}{2\ga^2 d d_2} \left(\frac{|k + 2(D_1 - D_2)|^2}{D_1} + \frac{k^2}{2\mu D_3}\right)\right] b^2 |\gw |.
\end{equation}
Moreover, for any given bounded set $B \subset H$ there exists a finite time $T_B^{p,\theta} \geq 0$ such that $\|p(t)\|^2 + \|\theta (t)\|^2 < R_2$ for any $g_0 \in B$ and all $t > T_B^{p,\theta}$. The combination of \eqref{vzsup} and \eqref{pthsup} gives rise to
\begin{equation} \label{msup}
	\begin{split}
    &\limsup_{t \to \infty} \,(\| u(t) - w(t) \|^2 + \|\vp(t) - \psi(t)\|^2) \\[2pt]
    = & \limsup_{t \to \infty} \, (\|p (t) - (v(t) - z(t))\|^2 + \|\vp(t) - \psi(t)\|^2) \\[2pt]
     \leq &\limsup_{t \to \infty} \left[3 (\|p(t)\|^2 + \|v(t)\|^2 + \|z(t)\|^2) + \|\theta (t)\|^2 \right] < 3 (R_0 + R_2).
    	\end{split}
\end{equation}

Finally, putting together \eqref{vzsup}, \eqref{psup} and \eqref{msup} and noticing \eqref{4com} we reach the conclusion that 
\begin{equation} \label{gsup}
	\begin{split}
	\limsup_{t \to \infty}& \,\|g(t)\|^2 = \limsup_{t \to \infty} \, \left[\frac{1}{4} \| (u + w) + (u - w)\|^2 + \frac{1}{4} \| (u + w) - (u - w)\|^2 \right. \\
	 &\left. + \frac{1}{4} \| (\vp + \psi) + (\vp - \psi)\|^2 + \frac{1}{4} \| (\vp + \psi) - (\vp - \psi)\|^2 + \|v\|^2 + \|z \|^2\right] \\
	 & \leq \limsup_{t \to \infty} \, \left(\|u + w\|^2 + \|u-w\|^2 + \|\vp + \psi\|^2 + \|\vp - \psi\|^2 + \|v\|^2 + \|z \|^2 \right) \\
	 & < 7R_0 + 3(R_1 + R_2),
	\end{split}
\end{equation}
and that for any given bounded set $B \subset H$ there exists a finite time 
$$
	T_B = \max \{T_B^{v,z}, T_B^{y,\xi}, T_B^{p,\theta}\} \geq 0
$$ 
such that $\|g (t)\|^2 < K_1 =  7R_0 + 3(R_1 + R_2)$ for any $g_0 \in B$ and all $t > T_B$. Therefore the lemma is proved with $K_1 =  7R_0 + 3(R_1 + R_2)$ in the description of an absorbing set $B_0$ in \eqref{bk} and $K_1$ is a constant independent of initial data.
\end{proof}

Next we show the absorbing properties of the $(v, z)$ components of the extended Brusselator semiflow in the product Banach spaces $[L^{2p} (\gw)]^2$, for any integer $1 \leq p \leq 3$. 

\begin{lemma} \label{L:absbp}
	For any given integer $1 \leq p \leq 3$, there exists a positive constant $K_p$ such that the absorbing inequality
\begin{equation} \label{lsupp}
	\limsup_{t \to \infty} \, \|(v(t), z(t))\|_{L^{2p}}^{2p} < K_p
\end{equation}
is satisfied by the $(v, z)$ components of the extended Brusselator semiflow $\{S(t)\}_{t \geq 0}$ for any initial data $g_0 \in H$. Moreover, for any bounded set $B \subset H$ there is a finite time $T_B^p > 0$ such that $\|(v(t), z(t))\|_{L^{2p}}^{2p} < K_p$ for any $g_0 \in B$ and all $t > T_B^p$.
\end{lemma}
\begin{proof}
The case $p = 1$ has been shown in Lemma \ref{L:absb}. According to the solution property \eqref{soln} with $T_{max} = \infty$ for all solutions, for any given initial status $g_0 \in H$ there exists a time $t_0 \in (0, 1)$ such that
\begin{equation} \label{tog}
	S(t_0) g_0  \in E = [H_0^1 (\gw)]^6 \hookrightarrow \mathbb{L}^6 (\gw) \hookrightarrow \mathbb{L}^4 (\gw).
\end{equation}
Then the weak solution $g(t) = S(t)g_0$ becomes a strong solution on $[t_0, \infty)$ and satisfies
\begin{equation} \label{esf}
	S(\cdot) g_0 \in C([t_0, \infty); E) \cap L^2 (t_0, \infty; \Pi) \subset C([t_0, \infty); \mathbb{L}^6 (\gw)) \subset C([t_0, \infty); \mathbb{L}^4 (\gw)), 
\end{equation}
for $n \leq 3$. Based on this observation, without loss of generarity, we can simply \emph{assume} that $g_0 \in \mathbb{L}^6 (\gw)$ for the purpose of studying the long-time dynamics. Thus parabolic regularity \eqref{esf} of strong solutions implies the $S(t)g_0 \in E \subset \mathbb{L}^6 (\gw), t \geq 0$. Then by the bootstrap argument, again without loss of generality, one can \emph{assume} that $g_0 \in \Pi \subset \mathbb{L}^8 (\gw)$ so that $S(t)g_0 \in \Pi \subset \mathbb{L}^8 (\gw), t \geq 0$.

Take the $L^2$ inner-product $\inpt{\eqref{eqv},v^5}$ and $\inpt{\eqref{eqz},z^5}$ and sum them up to obtain
\begin{equation} \label{vz6}
	\begin{split}
	&\frac{1}{6} \frac{d}{dt} \left(\sn{v(t)} + \sn{z(t)}\right) + 5d_2 \left(\|v(t)^2 \nabla v(t)\|^2 + \|z(t)^2 \nabla z(t)\|^2\right) \\
	= &\, \int_{\gw} \left(bu(t,x)v^5 (t,x) - u^2(t,x)v^6(t,x) + bw(t,x)z^5(t,x) - w^2(t, x) z^6(t,x)\right) dx\\
	& + D_2 \int_{\gw} \left[(z(t,x) - v(t,x))v^5(t,x) + (v(t,x) - z(t,x))z^5(t,x)\right] dx.
	\end{split}
\end{equation}
By Young's inequality, we have
\begin{equation*}
	\int_{\gw} \left[\left(buv^5 - u^2v^6\right) + \left(bwz^5 - w^2z^6\right)\right] dx \leq \frac{1}{2} \left(\int_{\gw} b^2 (v^4 + z^4)\, dx - \int_{\gw} (u^2 v^6 + w^2 z^6)\, dx \right),
\end{equation*}
and
\begin{equation*}
	 \int_{\gw} \left[(z - v)v^5 + (v - z)z^5\right] dx \leq  \int_{\gw}  \left[- v^6 + \left(\frac{1}{6} z^6 + \frac{5}{6} v^6 dx\right) + \left(\frac{1}{6} v^6 + \frac{5}{6} z^6 \right) - z^6 \right] dx = 0.
\end{equation*}
Substitute the above two inequalities into \eqref{vz6} and use Poincar\'{e} inequality, we get the following inequality relating $\sn{(v, z)}$ to $\fn{(v, z)}$,
\begin{equation}  \label{6vz}
	\begin{split}
	&\frac{d}{dt} \left(\sn{v(t)} + \sn{z(t)}\right) + \frac{10}{3}\ga d_2 \left(\sn{v(t)} + \sn{z(t)} \right) \\
	 \leq &\, \frac{d}{dt} \left(\sn{v(t)} + \sn{z(t)}\right) + \frac{10}{3} d_2 \left(\|\nabla v^3 (t)\|^2 + \| \nabla z^3(t)\|^2\right) \leq 3b^2 (\fn{v(t)} + \fn{z(t)}).
	\end{split}
\end{equation}
Similarly by taking the $L^2$ inner-product $\inpt{\eqref{eqv},v^3}$ and $\inpt{\eqref{eqz},z^3}$ we can get the corresponding inequality relating $\fn{(v, z)}$ to $\|(v, z)\|^2$,
\begin{equation}  \label{4vz}
	\begin{split}
	&\frac{d}{dt} \left(\fn{v(t)} + \fn{z(t)}\right) + 3\ga d_2 \left(\fn{v(t)} + \fn{z(t)} \right) \\
	 \leq &\, \frac{d}{dt} \left(\fn{v(t)} + \fn{z(t)}\right) + 3 d_2 \left(\|\nabla v^2 (t)\|^2 + \| \nabla z^2(t)\|^2\right) \leq 2b^2 (\|v(t)\|^2 + \|z(t)\|^2).
	\end{split}
\end{equation}
Applying Gronwall inequality to inequalities \eqref{4vz}, \eqref{6vz}, and using \eqref{vz}, we get
\begin{align*}
	& \fn{v(t)} + \fn{z(t)} \\
	\leq &\, e^{-3\ga d_2 t} \left(\fn{v_0} + \fn{z_0}\right) + \int_0^t e^{-3\ga d_2 (t-\tau)} 2b^2 (\|v(\tau)\|^2 + \|z(\tau)\|^2)d\tau \\
	\leq &\, e^{-3\ga d_2 t} \left(\fn{v_0} + \fn{z_0}\right) + \int_{\gw} e^{-3\ga d_2(t-\tau)-2\ga d_2 \tau}2b^2(\|v_0\|^2 + \|z_0\|^2)\, d\tau + \frac{b^4 |\gw |}{3\ga^2 d_2^2} \\
	 \leq &\, e^{- \ga d_2 t} C_{12} \left(\sn{v_0} + \sn{z_0}\right) + \frac{b^4 |\gw |}{3\ga^2 d_2^2}, \quad t \geq 0,
\end{align*}
where $C_{12}$ is a positive constant, and then
\begin{align*}
	&\sn{v(t)} + \sn{z(t)} \\
	\leq &\, e^{- (10/3) \ga d_2 t} \left(\sn{v_0} + \sn{z_0}\right) + \int_0^t e^{- (10/3) \ga d_2 (t-\tau)} 3b^2 (\fn{v(\tau)} + \fn{z(\tau)})d\tau \\
	\leq &\, e^{- 3 \ga d_2 t} \left(\sn{v_0} + \sn{z_0}\right) + \int_{\gw} e^{- 3 \ga d_2(t-\tau)-2\ga d_2 \tau}3b^2 C_{12}(\sn{v_0} + \sn{z_0})\, d\tau + \frac{3b^6 |\gw |}{20\ga^3 d_2^3} \\
	\leq &\, e^{-\ga d_2 t} C_{13} \left(\sn{v_0} + \sn{z_0}\right) + \frac{3b^6 |\gw |}{20\ga^3 d_2^3}, \quad t \geq 0,
\end{align*}
where $C_{13}$ is a positive constant. It follows that
\begin{align} 
	\limsup_{t \to \infty} \, &\left(\fn{v(t)} + \fn{z(t)}\right) < K_2 = 1 + \frac{b^4 |\gw |}{3\ga^2 d_2^2},  \label{vz4sup} \\
	\limsup_{t \to \infty} \, &\left(\sn{v(t)} + \sn{z(t)}\right) < K_3 = 1 + \frac{3b^6 |\gw |}{20\ga^3 d_2^3}. \label{vz6sup} 
\end{align}
Thus \eqref{lsupp} is proved and the last statement of this lemma is also proved.
\end{proof}

\section{\textbf{Asymptotic Compactness}}

In this section, we show that the extended Brusselator semiflow $\{S(t)\}_{t\geq 0}$ is asymptotically compact through the following two lemmas. Since $H_0^1 (\gw) \hookrightarrow L^4 (\gw)$ and $H_0^1 (\gw) \hookrightarrow L^6 (\gw)$ are continuous embeddings, there are constants $\ggd > 0$ and $\eta > 0$ such that $\| \cdot \|_{L^4}^2 \leq \ggd \|\nabla (\cdot)\|^2$ and $\| \cdot \|_{L^6}^2 \leq \eta \|\nabla (\cdot)\|^2$. We shall use the notation $\|(y_1, y_2)\|^2 = \|y_1\|^2 + \|y_2\|^2$ for conciseness.

\begin{lemma} \label{L:uwc}
	For any given initial data $g_0 \in B_0$, the $(u, w)$ components of the solution trajectories $g(t) = S(t)g_0$ of the IVP \eqref{eveq} satisfy 
\begin{equation} \label{uwq}
	\|\nabla (u(t), w(t)) \|^2 \leq Q_1, \quad \textup{for} \; \;  t > T_1,
\end{equation}
where $Q_1 > 0$ is a constant depending only on $K_1$ and $|\gw |$ but independent of initial data, and $T_1 > 0$ is finite and only depends on the absorbing ball $B_0$.
\end{lemma}

\begin{proof}
	Take the inner-products $\inpt{\eqref{equ}, -\gd u(t)}$ and $\inpt{\eqref{eqw}, -\gd w(t)}$ and then sum up the two equalities to obtain
	
\begin{align*}
	&\frac{1}{2} \frac{d}{dt} \|\nabla(u, w)\|^2 + d_1 \|\gd (u,w)\|^2 + (b+k)\|\nabla(u, w)\|^2 \\
	 = &\, - \int_\gw a (\gd u +\gd w)\, dx - \int_\gw (u^2 v \gd u + w^2 z \gd w)\, dx - N \int_\gw(\vp \gd u +\psi \gd w)\, dx \\
	&\, - D_1 \int_\gw  (|\nabla u |^2 - 2 \nabla u \cdot \nabla w + |\nabla w |^2 )\, dx \\
	 \leq &\left(\frac{d_1}{4} + \frac{d_1}{4} + \frac{d_1}{2}\right) \|\gd (u, w)\|^2 + \frac{a^2}{d_1} |\gw | + \frac{N^2}{d_1} \|(\vp, \psi )\|^2 + \frac{1}{2d_1} \int_\gw \left(u^4 v^2 + w^4 v^2 \right) dx.
\end{align*}
It follows that
\begin{equation} \label{nuw}
	\begin{split}
	&\frac{d}{dt} \|\nabla(u, w)\|^2 + 2(b+k)\|\nabla(u, w)\|^2 \\
	\leq &\, \frac{2}{d_1} \left(a^2 |\gw | + N^2 \|(\vp, \psi) \|^2\right) + \frac{1}{d_1} \left(\|u^2\|^2 \| v\|^2 + \|w^2\|^2 \| z\|^2 \right) \\
	\leq &\, \frac{2}{d_1} \left(a^2 |\gw | + N^2 \|(\vp, \psi) \|^2\right) + \frac{\ggd^2}{d_1} \left(\|v\|^2 \|\nabla u \|^4 + \|z\|^2 \|\nabla w \|^4 \right), \quad t > 0.
	\end{split}
\end{equation}
By the absorbing property shown in Lemma \ref{L:absb}, there is a finite time $T_0 = T_0 (B_0) \geq 0$ such that $S(t) B_0 \subset B_0$ for all $t > T_0$. Therefore, for any $g_0 \in B_0$, by \eqref{bk} we have 
\begin{equation} \label{vzvp}
	\| (u(t), w(t))\|^2 + \| (v(t), z(t))\|^2 + \|(\vp(t), \psi (t))\|^2 \leq K_1, \quad \textup{for} \;\; t > T_0.
\end{equation}
Substitute \eqref{vzvp} into \eqref{nuw} to obtain
\begin{equation} \label{nuwiq}
	\begin{split}
	\frac{d}{dt} \|\nabla (u, w)\|^2 &\leq \frac{d}{dt} \|\nabla (u, w)\|^2 + 2(b+k)\|\nabla(u, w)\|^2 \\[3pt]
	& \leq \frac{\ggd^2 K_1}{d_1}\|\nabla (u, w)\|^4 +  \frac{2}{d_1} \left(a^2 |\gw | + N^2 K_1\right), \quad t > T_0,
	\end{split}
\end{equation}
which can be written as the inequality
\begin{equation} \label{ugi}
	\frac{d \rho}{dt} \leq \alpha \rho + \frac{2}{d_1} \left(a^2 |\gw | + N^2 K_1\right), \quad t > T_0,
\end{equation}
where 
$$
	\rho (t) = \|\nabla (u(t), w(t))\|^2 \quad \textup{and} \quad \alpha (t) = \frac{\ggd^2 K_1}{d_1} \rho (t).
$$

In view of the inequality \eqref{nyiq}, \eqref{vztt} and \eqref{vzvp}, we have
\begin{equation} \label{ytt}
	\begin{split}
	&\int_t^{t+1} \|\nabla y(\tau)\|^2 \, d\tau \leq \frac{2|d_1 - d_2|^2}{d_1^2} \int_t^{t+1} \|\nabla (v + z)\|^2 d\tau + \frac{2k}{d_1} \int_t^{t+1} \|\Xi (\tau)\|^2 d\tau \\
	& + \frac{1}{d_1} \left(\|y(t)\|^2 + \int_t^{t+1} \frac{8}{\ga} (\|v(\tau)\|^2 +\|z(\tau)\|^2 + 2a^2 |\gw |)\, d\tau \right) \leq C_{14}, \;\; \textup{for} \;\;t > T_0, 
	\end{split}
\end{equation}
where
$$
	C_{14} = \frac{4|d_1 - d_2|^2}{d_1^2 d_2}\left[K_1 + \left(1 + \frac{1}{2\ga d_2}\right)b^2 |\gw |\right] + \frac{1}{d_1} \left((4k\mu^{-1} + 1) K_1 + \frac{8}{\ga}(K_1 + 2a^2 |\gw|)\right).
$$
From the inequality \eqref{npiq}, \eqref{vztt} and \eqref{vzvp} and with a similar estimation, there exists a constant $C_{15} > 0$ such that 
\begin{equation} \label{ptt}
	\int_t^{t+1} \|\nabla p(\tau)\|^2 \,d\tau \leq C_{15}, \quad \textup{for} \;\; t > T_0.
\end{equation}
According to \eqref{4com}, we can put together \eqref{vztt}, \eqref{ytt} and \eqref{ptt} to get
\begin{equation} \label{rtt}
	\begin{split}
	&\int_t^{t+1} \rho (\tau) \, d\tau  = \int_t^{t+1} (\|\nabla u(\tau)\|^2 + \|\nabla w(\tau)\|^2)\, d\tau \\
	\leq &\, \frac{1}{2} \int_t^{t+1} \left(\|\nabla (y(\tau) - (v(\tau) + z(\tau))\|^2 + \|\nabla (p(\tau) - (v(\tau) - z(\tau))\|^2 \right) d\tau \\
	\leq &\, \int_t^{t+1} \left(\|\nabla y(\tau)\|^2 + \|\nabla p(\tau)\|^2 + \|\nabla (v + z)\|^2 + \|\nabla (v - z)\|^2\right) d\tau \\
	\leq &\, C_{14} + C_{15} + \frac{4}{d_2} \left[K_1 + \left(1 + \frac{1}{2\ga d_2}\right)b^2 |\gw|\right] \overset{\textup{def}}{=} C_{16}, \quad \textup{for} \; \; t > T_0.
	\end{split}
\end{equation}
Now we can apply the uniform Gronwall inequality, cf. \cite[Lemma D.3]{SY02}, to \eqref{ugi} and use \eqref{rtt} to reach the conclusion \eqref{uwq} with
$$
	Q_1 = \left(C_{16} +  \frac{2}{d_1} \left(a^2 |\gw | + N^2 K_1\right)\right) e^{\ggd^2 K_1C_{16}/d_1}
$$
and $T_1 = T_0 (B_0) + 1$. The proof is completed.
\end{proof}

\begin{lemma} \label{L:vzvpc}
	For any given initial data $g_0 \in B_0$, the $(v, z)$ components and $(\vp, \psi)$ components of the solution trajectories $g(t) = S(t)g_0$ of the IVP \eqref{eveq} satisfy 
\begin{equation} \label{q2}
	\|\nabla (v(t), z(t)) \|^2 + \|\nabla (\vp(t), \psi(t)) \|^2 \leq Q_2, \quad \textup{for} \; \;  t > T_2,  
\end{equation}
where $Q_2 > 0$ is a constants depending on $K_1$ and $|\gw |$ but independent of initial data, and $T_2\, (> T_1 > 0)$ is finite and only depends on the absorbing ball $B_0$.
\end{lemma}

\begin{proof}
	Take the inner-products $\inpt{\eqref{eqv}, -\gd v(t)}$ and $\inpt{\eqref{eqz}, -\gd z(t)}$ and then sum up the two equalities to obtain
\begin{equation} \label{vzq}
	\begin{split}
	\frac{1}{2} &\frac{d}{dt} \|\nabla (v, z)\|^2 + d_2 \|\gd (v, z)\|^2 = -\int_\gw b(u\gd v + w \gd z)\, dx \\
	& + \int_\gw (u^2 v \gd v + w^2 z \gd z )\, dx - D_2 \int_\gw [(z - v)\gd v + (v - z) \gd z ]\, dx \\
	 \leq &\, \frac{d_2}{2} \|\gd (v, z)\|^2 + \frac{b^2}{d_2} \|(u, w)\|^2 + \frac{1}{d_2} \int_\gw (u^4 v^2 + w^4 z^2) \, dx \\
	& - D_2 \int_\gw  (|\nabla v |^2 - 2 \nabla v \cdot \nabla z + |\nabla z |^2 )\, dx \\
	\leq &\, \frac{d_2}{2} \|\gd (v, z)\|^2 + \frac{b^2}{d_2} \|(u, w)\|^2 + \frac{1}{d_2} \int_\gw (u^4 v^2 + w^4 z^2) \, dx, \quad t > 0.
	\end{split}
\end{equation}
Since 
$$
	\|\nabla (v(t), z(t))\|^2 = - (\inpt{v, \gd v} + \inpt{z, \gd z}) \leq \frac{1}{2} \left(\|v(t)\|^2 +\|z(t)\|^2 + \|\gd v(t)\|^2 + \|\gd z(t)\|^2 \right),
$$
in \eqref{vzq} we have
$$
	\frac{d_2}{2} \|\gd (v, z)\|^2 \geq d_2 \|\nabla (v, z)\|^2 - \frac{d_2}{2} \| (v, z) \|^2.
$$
Then by using H\"{o}lder inequality and the embedding inequalities mentioned in the beginning of this section and by Lemma \ref{L:uwc}, from the above inequality \eqref{vzq} we get 
\begin{align*} 
	&\frac{d}{dt}\|\nabla (v, z)\|^2 \leq \frac{d}{dt}\|\nabla (v, z)\|^2 + 2d_2 \|\nabla (v, z)\|^2  \\
	\leq &\, d_2 \|(v, z)\|^2+ \frac{2b^2}{d_2} \|(u, w)\|^2 + \frac{2}{d_2} (\|u\|_{L^6}^4 \|v\|_{L^6}^2 + \|w\|_{L^6}^4 \|z\|_{L^6}^2 ) \\
	\leq &\, \left(d_2 + \frac{2b^2}{d_2}\right)K_1 + \frac{2\eta^6}{d_2} (\|\nabla u \|^4 + \|\nabla w \|^4)\|\nabla (v, z)\|^2 \\
	\leq &\, K_1 \left(d_2 + \frac{2b^2}{d_2} \right) + \frac{2\eta^6 Q_1^2}{d_2} \|\nabla (v, z)\|^2, \quad t > T_1.
\end{align*}
Applying the uniform Gronwall inequality to 
\begin{equation} \label{nbvz} 
	\frac{d}{dt}\|\nabla (v, z)\|^2 \leq \frac{2\eta^6 Q_1^2}{d_2} \|\nabla (v, z)\|^2 + K_1 \left(d_2 + \frac{2b^2}{d_2} \right), \quad t > T_1,
\end{equation}
and using \eqref{vztt}, we can assert that
\begin{equation} \label{nvziq}
	\|\nabla (v(t), z(t))\|^2 \leq C_{17}, \quad \textup{for} \;\; t > T_1 + 1,
\end{equation}
where
$$
	C_{17} = \left(\frac{1}{d_2} \left[K_1 + \left(1 + \frac{1}{2\ga d_2}\right) b^2 |\gw |\right] +  K_1 \left[d_2 + \frac{2b^2}{d_2} \right]\right) e^{2 \eta^6 Q_1^2/d_2} .
$$

Next take the inner-products $\inpt{\eqref{eqp}, \vp(t)}$ and $\inpt{\eqref{eqs}, \psi (t)}$ and then sum up the two equalities to get
$$
	\frac{1}{2} \frac{d}{dt} \|(\vp, \psi)\|^2 + d_3 \|\nabla (\vp, \psi)\|^2 + (\gl + N)\|(\vp, \psi)\|^2 \leq \frac{k}{2} \left(\|(u, w)\|^2 + \|(\vp, \psi)\|^2\right),
$$
so that
\begin{equation} \label{vptt}
	\begin{split}
	&\int_t^{t+1} \|\nabla (\vp(\tau), \psi(\tau))\|^2 \, d\tau \leq \frac{1}{2d_3} \|(\vp (t), \psi (t))\|^2 \\
         + \frac{k}{2d_3} \int_t^{t+1} &(\|(u(\tau), w(\tau))\|^2 + \|(\vp(\tau), \psi(\tau))\|^2) d\tau \leq \frac{(1 + k)K_1}{2d_3}, \quad t > T_0.
	\end{split}
\end{equation}
Then take the inner-products $\inpt{\eqref{eqp}, -\gd \vp(t)}$ and $\inpt{\eqref{eqs}, -\gd \psi (t)}$ and sum up the two equalities to obtain
\begin{align*}
	&\frac{1}{2} \frac{d}{dt} \|\nabla (\vp, \psi)\|^2 + d_3 \|\gd (\vp, \psi)\|^2 + (\gl + N) \|\nabla (\vp, \psi)\|^2 \\
	 = &\, - \int_\gw k(u \gd \vp + w \gd \psi)\, dx - D_3 \int_\gw [(\psi - \vp) \gd \vp + (\vp - \psi) \gd \psi ]\, dx \\
	 \leq &\, d_3 \|\gd (\vp, \psi)\|^2 + \frac{k^2}{4d_3} (\| u \|^2 + \| w \|^2 ) - D_3 \int_\gw (|\nabla \vp |^2 - 2 \nabla \vp \cdot \nabla \psi + |\nabla \psi |^2)\, dx \\
	 \leq &\, d_3 \|\gd (\vp, \psi)\|^2 + \frac{k^2}{4d_3} K_1, \quad t > T_0,
\end{align*}
so that
\begin{equation} \label{phps}
	\frac{d}{dt} \|\nabla (\vp, \psi)\|^2 + 2 (\gl + N) \|\nabla (\vp, \psi)\|^2 \leq \frac{k^2}{2d_3} K_1, \quad t > T_0.
\end{equation}
We can apply the uniform Gronwall inequality to \eqref{phps} with the aid of \eqref{vptt} to reach the estimate
\begin{equation} \label{nps}
	\|\nabla (\vp (t), \psi (t))\|^2 \leq \frac{K_1}{2d_3} (1 + k + k^2) e^{2(\gl + N)} \overset{\textup{def}}{=} C_{18}, \quad t > T_0 + 1.
\end{equation}
Finally let $Q_2 = C_{17} + C_{18}$ and $T_2 = T_1 + 1$. By \eqref{nvziq} and \eqref{nps}, we see that \eqref{q2} is proved.
\end{proof}

\section{\textbf{Global Attractor and Its Properties}}

In this section we reach the proof of Theorem \ref{Mthm} (Main Theorem) on the existence of a global attractor, which will be denoted by $\ms{A}$, for the extended Brusselator semiflow $\csg$ and we shall show several properties of this global attractor $\ms{A}$, namely, the regularity of $\ms{A}$, the property of being an $(H, E)$ global attractor, and the finite Hausdorff and fractal dimensionality.

\begin{proof}[Proof of Theorem \textup{1}]
In Lemma \ref{L:absb} we have shown that the extended Brusselator semiflow $\csg$ has an absorbing set $B_0$ in $H$. Combining Lemma \ref{L:uwc} and Lemma \ref{L:vzvpc} we proved that 
$$
	\|S(t)g_0 \|_E^2 \leq Q_1 + Q_2, \quad \textup{for} \;\; t > T_2 \;\; \textup{and for} \;\; g_0 \in B_0,
$$
which implies that $\{S(t)B_0: t > T_2\}$ is a bounded set in $E$ and consequently a precompact set in $H$. Therefore, the extended Brusselator semiflow ${\csg}$ is asymptotically compact in $H$.  Finally we apply Proposition \ref{P2} to reach the conclusion that there exists a global attractor $\ms{A}$ in $H$ for this extended Brusselator semiflow $\csg$.
\end{proof}

Now we show that the global attractor $\ms{A}$ of the extended Brusselator semiflow is an $(H, E)$ global attractor with the regularity $\ms{A} \subset \mathbb{L}^\infty (\gw)$. The concept of $(H, E)$ global attractor was introduced in \cite{BV83}.

\begin{definition} \label{D:hea}
	Let $\{\Sigma (t)\}_{t\geq 0}$ be a semiflow on a Banach space $X$ and let $Y$ be a compactly imbedded subspace of $X$. A subset $\mathcal{A}$ of $Y$ is called an $(X, Y)$ global attractor for this semiflow if $\mathcal{A}$ has the following properties,
	
	\textup{(i)} $\mathcal{A}$ is a nonempty, compact, and invariant set in $Y$.
	
	\textup{(ii)} $\mathcal{A}$ attracts any bounded set $B \subset X$ with respect to the $Y$-norm, namely, there is a time $\tau = \tau (B)$ such that $\Sigma (t)B \subset Y$ for $t > \tau$ and $dist_Y (\Sigma (t)B, \mathcal{A}) \to 0$, as $t \to \infty$.
\end{definition}

\begin{lemma} \label{L:ste}
	Let $\{g_m\}$ be a sequence in $E$ such that $\{g_m\}$ converges to $g_0 \in E$ weakly in $E$ and $\{g_m\}$ converges to $g_0$ strongly in $H$, as $m \to \infty$. Then
	$$
		\lim_{m \to \infty} S(t) g_m = S(t) g_0 \; \; \textup{strongly in} \; E,
	$$
and the convergence is uniform with respect to $t$ in any given compact interval $[t_0, t_1] \subset (0, \infty)$.
\end{lemma}

The proof of this lemma is seen in \cite[Lemma 10]{yY10}.

\begin{theorem} \label{Thm2}
	The global attractor $\ms{A}$ in $H$ for the extended Brusselator semiflow $\csg$ is indeed an $(H, E)$ global attractor and $\ms{A}$ is a bounded subset in $\mathbb{L}^\infty (\gw)$.
\end{theorem}

\begin{proof}
	By Lemma \ref{L:uwc} and Lemma \ref{L:vzvpc}, we can assert that there exists a bounded absorbing set $B_1 \subset E$ for the extended Brusselator semiflow $\csg$ on $H$ and this absorbing is in the $E$-norm. Indeed,
$$
	B_1 = \{g \in E: \| g \|_E^2 = \|\nabla g \|^2 \leq Q_1 + Q_2\}.
$$
	
	Now we show that this semiflow $\csg$ is asymptotically compact with respect to the strong topology in $E$. For any time sequence $\{t_n \}, t_n \to \infty$, and any bounded sequence $\{g_n \} \subset E$, there exists a finite time $t_0 \geq 0$ such that $S(t) \{g_n\} \subset B_0$, for any $t > t_0$. Then for an arbitrarily given $T > t_0 + T_2$, where $T_2$ is the time specified in Lemma \ref{L:vzvpc}, there is an integer $n_0 \geq 1$ such that $t_n > 2T$ for all $n > n_0$. According to Lemma \ref{L:uwc} and Lemma \ref{L:vzvpc}, 
$$
	\{S(t_n - T) g_n\}_{n > n_0} \; \textup{is a bounded set in}\; E.
$$	
Since $E$ is a Hilbert space, there is an increasing sequence of integers $\{n_j\}_{j=1}^\infty$ with $n_1 > n_0$, such that
$$
	  \lim_{j \to \infty} S(t_{n_j} - T) g_{n_j} = g^* \;\; \textup{weakly in} \; E.
$$
By the compact imbedding $E \hookrightarrow H$, there is a further sebsequence of $\{n_j\}$, but relabeled as the same as $\{n_j\}$, such that
$$
	\lim_{j \to \infty} S(t_{n_j} - T) g_{n_j} = g^* \;\; \textup{strongly in} \; H.
$$
Then by Lemma \ref{L:ste}, we have the following convergence with respect to the $E$-norm,
$$
	\lim_{j \to \infty} S(t_{n_j}) g_{n_j} = \lim_{j \to \infty} S(T) S(t_{n_j} - T) g_{n_j} = S(T) g^* \;\; \textup{strongly in} \; E.
$$
This proves that $\csg$ is asymptotically compact in $E$. 

Therefore, by Proposition \ref{P2}, there exists a global attractor $\ms{A}_E$ for the extended Brusselator semiflow $\csg$ in $E$. According to Definition \ref{D:hea} and the fact that $B_1$ attracts $B_0$ in the $E$-norm due to the combination of Lemma \ref{L:uwc} and Lemma \ref{L:vzvpc}, we see that this global attractor $\ms{A}_E$ is an $(H, E)$ global attractor. Moreover, the invariance and the boundedness of $\ms{A}$ in $H$ and in $E$ imply that 
\begin{align*}
	&\ms{A}_E \; \textup{attracts} \; \ms{A} \; \textup{in} \; E, \; \textup{so that} \; \ms{A} \subset \ms{A}_E; \\
	&\ms{A} \; \textup{attracts} \; \ms{A}_E \; \textup{in} \; H, \; \textup{so that} \; \ms{A}_E \subset \ms{A}.
\end{align*}
Therefore, $\ms{A} = \ms{A}_E$ and we proved that the global attractor $\ms{A}$ in $H$ is indeed an $(H, E)$ global attractor for this extended Brusselator semiflow.

Next we show that $\ms{A}$ is a bounded subset in $\mathbb{L}^\infty (\gw)$. By the $(L^p, L^\infty)$ regularity of the analytic $C_0$-semigroup $\{e^{At}\}_{t\geq 0}$ stated in \cite[Theorem 38.10]{SY02}, one has $e^{At}: \mathbb{L}^p (\gw) \longrightarrow \mathbb{L}^\infty (\gw)$ for $t > 0$, and there is a constant $C(p) > 0$ such that
\begin{equation} \label{cp}
	\| e^{At} \|_{\mathcal{L} (\mathbb{L}^p, \mathbb{L}^\infty)} \leq C(p) \, t^{- \frac{n}{2p}}, \;\; t > 0, \; \; \textup{where} \; n = \textup{dim} \, \gw.
\end{equation}
By the variation-of-constant formula satisfied by the mild solutions (of course strong solutions), for any $g \in \ms{A} \, (\subset E)$, we have
\begin{equation} \label{mld}
	\begin{split}
	\|S(t) g\|_{L^\infty} &\leq \|e^{At} \|_{\mathcal{L} (L^2,L^\infty)} \|g\| + \int_0^t \|e^{A(t- \sigma)}\|_{\mathcal{L} (L^2, L^\infty)} \| f(S(\sigma)g) \| \, d\sigma \\
	&\leq  C(2) t^{- \frac{3}{4}} \|g\| + \int_0^t C(2) (t- \sigma)^{-\frac{3}{4}} L(Q_1, Q_2) \|S(\sigma) g\|_E \, d\sigma, \quad t \geq 0,
	\end{split}
\end{equation}
where $C(2)$ is in the sense of \eqref{cp}, and $L(Q_1, Q_2)$ is the Lipschitz constant of the nonlinear map $f$ restricted on the closed, bounded ball centered at the origin with radius $Q_1 + Q_2$ in $E$. By the invariance of the global attractor $\ms{A}$, surely we have
$$
	\{S(t) \ms{A}: t \geq 0\} \subset B_0 \, (\subset H) \quad \textup{and} \quad \{S(t) \ms{A}: t \geq 0\} \subset B_1 \, (\subset E).
$$
Then from \eqref{mld} we get
\begin{equation} \label{bdift}
	\begin{split}
	\|S(t) g\|_{L^\infty} &\leq C(2) K_1 t^{- \frac{3}{4}} + \int_0^t C(2) L(Q_1, Q_2) (Q_1 + Q_2) (t- \sigma)^{-\frac{3}{4}} \, d\sigma \\
	&= C(2) [K_1 t^{- \frac{3}{4}} +  4 L(Q_1, Q_2) (Q_1 + Q_2) t^{\frac{1}{4}}], \quad \textup{for} \; t > 0.
	\end{split}
\end{equation}
Specifically one can take $t = 1$ in \eqref{bdift} and use the invariance of $\ms{A}$ to obtain
$$
	\| g \|_{L^\infty} \leq C(2)( K_1 +  4 L(Q_1, Q_2) (Q_1 + Q_2)), \quad \textup{for any} \; g \in \ms{A}.
$$
Thus the global attractor $\ms{A}$ is a bounded subset in $\mathbb{L}^\infty (\gw)$.
\end{proof}

Next consider the Hausdorff and fractal dimensions of the global attractor $\ms{A}$. The background concepts and results can be seen in \cite[Chapter V]{rT88}. Let $q_{m} = \limsup_{t \to \infty} \, q_{m} (t)$, where
\begin{equation} \label{trq} 
	q_{m} (t) = \sup_{g_0 \in \ms{A}} \;  \, \sup_{\substack{g_{i} \in H, \|g_{i} \| = 1\\ i = 1, \cdots, m}} \; \, \left( \frac{1}{t} \int_{0}^{t} \textup{Tr} \left[ \left( A + f^{\prime} (S(\tau) g_0 ) \right) \Gamma_{m} (\tau)\right] \, d\tau \right),
\end{equation}
in which $\textup{Tr} \, [(A + f^{\prime} (S(\tau)g_0)) \Gamma_m (\tau)]$ is the trace of the linear operator $(A + f^{\prime} (S(\tau)g_0))\Gamma_m (\tau)$, $f^{\prime} (g)$ is the Fr\'{e}chet derivative of the Nemytskii map $f$ in \eqref{eveq}, and $\Gamma_{m} (t)$ stands for the orthogonal projection of the space $H$ on the subspace spanned by $G_1 (t), \cdots, G_{m} (t)$, with
\begin{equation} \label{Frder}
	G_{i} (t) = L(S(t), g_0)g_{i},   \quad i = 1, \cdots, m.
\end{equation}
Here $f^{\prime}(S(\tau)g_0)$ is the Fr\'{e}chet derivative of the map $f$ defined by \eqref{opF} at  $S(\tau)g_0$, and $L(S(t), g_0)$ is the Fr\'{e}chet derivative of the map $S(t)$ at $g_0$, with $t$ fixed. 

The following proposition, cf.  \cite[Chapter 5]{rT88},  will be used to show the finite upper bounds of the Hausdorff and fractal dimensions of this global attractor $\ms{A}$.

\begin{proposition} \label{P3}
If there is an integer $m$ such that $q_{m} < 0$, then the Hausdorff dimension $d_{H} (\ms{A})$ and the fractal dimension $d_{F} (\ms{A})$ of $\ms{A}$ satisfy
\begin{equation}  \label{hfd}
	d_{H} (\ms{A}) \leq m,  \quad \textup{and} \quad d_{F} (\ms{A}) \leq m \max_{1 \leq j \leq m - 1} \left( 1 + \frac{(q_{j})_{+}}{| q_{m} |} \right) \leq 2m. 
\end{equation}
\end{proposition}
It can be shown that for any given $t > 0$, $S(t)$ is Fr\'{e}chet differentiable in $H$ and uniformly Fr\'{e}chet differentiable in $\ms{A}$. Its Fr\'{e}chet derivative at $g_0$ is the bounded linear operator $L(S(t), g_0)$ given by 
$$
	L(S(t),g_0)G_0 \overset{\textup{def}}{=} G(t) = (U(t), V(t), \Phi (t), W(t), Z(t), \Psi (t)), 
$$
for $G_0 = (U_0, V_0, \Phi_0, W_0, Z_0, \Psi_0) \in H$, where $(U(t), V(t), \Phi (t), W(t), Z(t), \Psi(t))$ is the weak solution of the following extended Brusselator variational equation

\begin{equation} \label{vareq}
	\begin{split}
	\frac{\partial U}{\partial t} & = d_1 \gd U + 2u(t)v(t) U + u^2 (t) V - (b + k) U + D_1 (W - U),  \\
	\frac{\partial V}{\partial t} & = d_2 \gd V - 2u(t)v(t) U - u^2 (t) V + b U + D_2 (Z - V),  \\
	\frac{\partial \Phi}{\partial t} & = d_3 \gd \Phi + k U - (\gl + N)\Phi + D_3 (\Psi - \Phi), \\
	\frac{\partial W}{\partial t} & = d_1 \gd W + 2w(t)z(t) W + w^2 (t) Z - (b + k) W + D_1 (U - W),  \\
	\frac{\partial Z}{\partial t} & = d_2 \gd Z - 2w(t)z(t) W - w^2 (t) Z + b W + D_2 (V - Z),\\
	\frac{\partial \Psi}{\partial t} & = d_3 \gd \Psi + k W - (\gl + N)\Psi + D_3 (\Phi - \Psi),  \\[3pt]
	U(0) &= U_0, \;\; V(0) = V_0, \;\; \Phi(0) = \Phi_0, \; \;  W(0) = W_0, \;\; Z(0) = Z_0 \;\; \Psi(0) = \Psi_0. 
	\end{split}
\end{equation}
Here $g(t) = (u(t), v(t), \varphi (t), w(t), z(t), \psi (t)) = S(t)g_0$ is the weak solution of \eqref{eveq} with the initial condition $g(0) = g_0$.  The initial value problem \eqref{vareq} can be written as 
\begin{equation} \label{vareveq}
	\frac{dG}{dt} = (A + f^{\prime} (S(t)g_0))G,   \quad G(0) = G_0.
\end{equation}

As we have shown, the invariance of $\ms{A}$ implies $\ms{A} \subset B_0 \cap B_1$, so that 
\begin{equation} \label{gbe}
	\sup_{g_0 \in \ms{A}} \, \| S(t)g_0 \|^2 \leq K_1 \quad \textup{and} \quad \sup_{g_0 \in \ms{A}} \, \| S(t)g_0 \|_E^2 \leq Q_1 + Q_2.
\end{equation}

\begin{theorem} \label{Dmn}
The global attractors $\ms{A}$ for the extended Brusselator semiflow $\csg$ has a finite Hausdorff dimension and a finite fractal dimension.
\end{theorem}

\begin{proof}
By Proposition \ref{P3}, we shall estimate $\textup{Tr} \, [(A + f^{\prime} (S(\tau)g_0 )) \Gamma_{m}(\tau)]$. At any given time $\tau > 0$, let $\{\zeta_{j} (\tau): j = 1, \cdots , m\}$ be an $H$-orthonormal basis for the subspace 
$$
	\Gamma_m(\tau) H = \textup{Span}\,  \{G_1(\tau), \cdots , G_,(\tau) \},
$$
where $G_1 (t), \cdots , G_{m} (t)$ satisfy \eqref{vareveq} and, without loss of generality, assuming that the initial vectors $G_{1,0}, \cdots , G_{m,0}$ are linearly independent in $H$. By the Gram-Schmidt orthogonalization scheme,  $\zeta_{j}(\tau) \in E$ and $\zeta_{j} (\tau)$ is strongly measurable in $\tau, j = 1, \cdots , m$. Let $d_0 = \min \{d_1, d_2, d_3\}$. Denote the components of $\zeta_j (\tau)$ by $\zeta_j^i (\tau), i = 1, \cdots, 6$. Then we have
\begin{equation} \label{Trace}
	\begin{split}
	\textup{Tr} \, [(A + f^{\prime} (S(\tau)g_0 ) \Gamma_m(\tau)]& = \sum_{j=1}^{m} \left( \langle A \zeta_{j}(\tau), \zeta_{j}(\tau) \rangle + \langle  f^{\prime} (S(\tau)g_0 ) \zeta_{j}(\tau), \zeta_{j}(\tau) \rangle\right)  \\[4pt]
	& \leq - d_0 \sum_{j=1}^{m} \, \| \nabla \zeta_{j}(\tau) \|^2 + J_1 + J_2 + J_3,  
	\end{split}
\end{equation}
where
\begin{align*}
	J_1  =&\, \sum_{j=1}^{m} \int_{\gw} 2 u(\tau) v(\tau) \left( |\zeta_{j}^1 (\tau) |^2 -  \zeta_{j}^1 (\tau) \zeta_{j}^2 (\tau)  \right) dx \\[4pt]
	&+ \sum_{j=1}^{m} \, \int_{\gw} 2 w(\tau) z(\tau) \left( |\zeta_{j}^4 (\tau) |^2 -  \zeta_{j}^4 (\tau) \zeta_{j}^5 (\tau)  \right) dx,
\end{align*}
\begin{align*}
	J_2 &= \sum_{j=1}^{m} \int_{\gw} \left( u^2(\tau) \left( \zeta_{j}^1 (\tau) \zeta_{j}^2 (\tau) - | \zeta_{j}^2 (\tau) |^2 \right) + w^2(\tau) \left( \zeta_{j}^4 (\tau) \zeta_{j}^5 (\tau) - | \zeta_{j}^5 (\tau) |^2 \right)\right) dx  \\[4pt]
	&\leq  \sum_{j=1}^{m} \int_{\gw} \left( u^2(\tau) | \zeta_{j}^1 (\tau) | |\zeta_{j}^2 (\tau)| + w^2(\tau) | \zeta_{j}^4 (\tau) | |\zeta_{j}^5 (\tau)| \right) dx,
\end{align*}
and

\begin{align*}
	J_3  = &\,  \sum_{j=1}^{m} \int_{\gw} \left( - (b + k) (|\zeta_{j}^1 (\tau) |^2 + |\zeta_{j}^4 (\tau) |^2) + b (\zeta_{j}^1 (\tau) \zeta_{j}^2 (\tau) + \zeta_{j}^4 (\tau) \zeta_{j}^5 (\tau)) \right) dx \\
	& +  \sum_{j=1}^{m} \int_{\gw} \left(k(\zeta_j^1 (\tau)\zeta_j^3 (\tau) + \zeta_j^4 (\tau)\zeta_j^6 (\tau)) - (\gl + N)(|\zeta_j^3 (\tau)|^2 + |\zeta_j^6 (\tau)|^2)\right) dx \\
	& - \sum_{j=1}^{m} \int_{\gw} \left(D_1 \left(\zeta_j^1 (\tau) - \zeta_j^4 (\tau)\right)^2 + D_2 \left(\zeta_j^2 (\tau) - \zeta_j^5 (\tau)\right)^2 + D_3 \left(\zeta_j^3 (\tau) - \zeta_j^6 (\tau)\right)^2\right) dx \\
	\leq &\, \sum_{j=1}^{m} \left(\int_{\gw} \, b \left(\zeta_{j}^1 (\tau) \zeta_{j}^2 (\tau) + \zeta_{j}^3 (\tau) \zeta_{j}^4 (\tau) \right) dx + \int_{\gw} k(\zeta_j^1 (\tau)\zeta_j^3 (\tau) + \zeta_j^4 (\tau)\zeta_j^6 (\tau))dx \right).
\end{align*}
By the generalized H\"{o}lder inequality, the embedding $H_0^1 (\gw) \hookrightarrow  L^4 (\gw)$ (for $n \leq 3$) and \eqref{gbe}, we get
\begin{equation} \label{J1eq}
		\begin{split}
	J_1 \leq &\, 2 \sum_{j=1}^{m} \| u(\tau) \|_{L^4} \| v(\tau) \|_{L^4}  \left( \| \zeta_{j}^1 (\tau) \|_{L^4}^2 + \| \zeta_{j}^1(\tau) \|_{L^4}  \| \zeta_{j}^2 (\tau) \|_{L^4} \right) \\
	& + 2 \sum_{j=1}^{m} \| w(\tau) \|_{L^4} \| z(\tau) \|_{L^4}  \left( \| \zeta_{j}^4 (\tau) \|_{L^4}^2 + \| \zeta_{j}^4(\tau) \|_{L^4}  \| \zeta_{j}^5 (\tau) \|_{L^4} \right) \\
	\leq &\, 4 \sum_{j=1}^{m} \| S(\tau)g_0 \|_{L^4}^2  \| \zeta_{j} (\tau) \|_{L^4}^2 \leq 4 \delta \sum_{j=1}^{m} \| \nabla S(\tau)g_0 \|^2   \| \zeta_{j} (\tau) \|_{L^4}^2 \\
	\leq &\, 4 \delta (Q_1 + Q_2)  \sum_{j=1}^{m} \| \zeta_{j} (\tau) \|_{L^4}^2, 
		\end{split}
\end{equation}
where $\delta$ is the embedding coefficient given in the beginning of Section 4. Now we use the Garliardo-Nirenberg interpolation inequality for Sobolev spaces \cite[Theorem B.3]{SY02},
\begin{equation} \label{GNineq}
	\| \zeta \|_{W^{k,p}} \leq C \| \zeta\|_{W^{m,q}}^{\theta} \| \zeta \|_{L^{r}}^{1 - \theta}, \quad \textup{for} \; \zeta \in W^{m,q}(\gw),
\end{equation}
provided that $p, q, r \geq 1, 0 < \theta \leq 1$, and
$$
	k - \frac{n}{p} \leq \theta \left( m - \frac{n}{q} \right)  - (1 - \theta ) \frac{n}{r},   \quad \textup{where} \; \, n = \textup{dim} \, \gw.
$$
Here with $W^{k, p}(\gw) = L^4(\gw), W^{m, q}(\gw) = H_{0}^{1}(\gw), L^{r}(\gw) = L^2(\gw)$, and $\theta = n/4 \leq 3/4$, it follows from \eqref{GNineq} that
\begin{equation} \label{inter}
	\| \zeta_{j} (\tau) \|_{L^4} \leq C \| \nabla \zeta_{j} (\tau) \|^{\frac{n}{4}} \| \zeta_{j} (\tau) \|^{1 - \frac{n}{4}} = C \| \nabla \zeta_{j} (\tau) \|^{\frac{n}{4}}, \quad j = 1, \cdots , m,
\end{equation}
since $\| \zeta_{j}(\tau) \| = 1$, where $C$ is a universal constant. Substitute \eqref{inter} into \eqref{J1eq} to obtain
\begin{equation} \label{J1est}
	J_1 \leq 4\delta (Q_1 + Q_2) C^2 \sum_{j=1}^{m} \, \| \nabla \zeta_{j} (\tau) \|^{\frac{n}{2}}.
\end{equation}
Similarly, we can get
\begin{equation} \label{J2est}
	J_2 \leq \delta (Q_1 + Q_2) \sum_{j=1}^{m} \| \zeta_{j} (\tau) \|_{L^4}^2 \leq \delta (Q_1 + Q_2) C^2 \sum_{j=1}^{m} \, \| \nabla \zeta_{j} (\tau) \|^{\frac{n}{2}}.
\end{equation}
Moreover, we have 
\begin{equation} \label{J3est}
	J_3 \leq \sum_{j=1}^{m}  (b + k) \| \zeta_{j} (\tau) \|^2 = m(b + k).
\end{equation}
Substituting \eqref{J1est}, \eqref{J2est} and \eqref{J3est} into \eqref{Trace}, we obtain
\begin{equation} \label{Trest}
	\begin{split}
	\textup{Tr} \, [(A + f^{\prime} (S(\tau)g_0 ) \Gamma_m(\tau)] \leq &\, - d_0 \sum_{j=1}^{m} \| \nabla \zeta_{j}(\tau) \|^2 \\
	&+ 5\delta (Q_1 + Q_2) C^2 \sum_{j=1}^{m} \| \nabla \zeta_{j}(\tau) \|^{\frac{n}{2}} + m(b + k).
	\end{split}
\end{equation}
By Young's inequality, for $n \leq 3$, we have
$$
	5 \delta (Q_1 + Q_2) C^2 \sum_{j=1}^{m} \| \nabla \zeta_{j}(\tau) \|^{\frac{n}{2}}  \leq \frac{d_0}{2} \sum_{j=1}^{m} \|\nabla \zeta_{j}(\tau) \|^2 + m \, Q_3(n),
$$
where $Q_3(n)$ is a universal positive constant depending only on $n = $ dim $(\gw)$. Hence,
\begin{equation*} 
	\textup{Tr} \, [(A + f^{\prime} (S(\tau)g_0 ) \Gamma_m(\tau)] \leq - \frac{d_0}{2} \sum_{j=1}^{m} \| \nabla \zeta_{j}(\tau) \|^2 + m ( Q_3 (n) + b + k), \quad \tau > 0, \; g_0 \in \ms{A}. 
\end{equation*}
According to the generalized Sobolev-Lieb-Thirring inequality \cite[Appendix, Corollary 4.1]{rT88}, since $\{ \zeta_1 (\tau), \cdots , \zeta_{m} (\tau) \}$ is an orthonormal set in $H$, there exists a universal constant $Q^* > 0$ only depending on the shape and dimension of $\gw$, such that

$$
	\sum_{j=1}^{m} \| \nabla \zeta_{j}(\tau) \|^2  \geq Q^* \frac{m^{1 + \frac{2}{n}}}{|\gw |^{\frac{2}{n}}}.
$$
Therefore, 
\begin{equation} \label{finest}
	\textup{Tr} \, [(A + f^{\prime} (S(\tau)g_0 ) \Gamma_m(\tau)] \leq - \frac{d_0 Q^*}{2 |\gw |^{\frac{2}{n}}} m^{1 + \frac{2}{n}} + m (Q_3 (n) + b + k), \quad \tau > 0,\; g_0 \in \ms{A}.
\end{equation}
Then we can conclude that
\begin{equation} \label{qm}
	\begin{split}
	q_m & = \limsup_{t \to \infty} \, q_m (t) \\
	& =  \limsup_{t \to \infty} \sup_{g_0 \in \ms{A}} \;  \, \sup_{\substack{g_{i} \in H, \|g_{i} \| = 1\\ i = 1, \cdots, m}} \; \, \left( \frac{1}{t} \int_{0}^{t} \textup{Tr}  \left[\left( A + f^{\prime} (S(\tau) g_0 ) \right) \Gamma_{m} (\tau)\right] \, d\tau \right)  \\
	& \leq - \, \frac{d_0 Q^*}{2 |\gw |^{\frac{2}{n}}} m^{1 + \frac{2}{n}}  + m (Q_3(n) + b + k) < 0,
	\end{split}
\end{equation}
if the integer $m$ satisfies the following condition,
\begin{equation} \label{dimc}
	m - 1 \leq \left( \frac{2(Q_3(n) + b + k)}{d_0 Q^*} \right)^{n/2}| \gw | < m.
\end{equation}
According to Proposition \ref{P3}, we have proved that the Hausdorff dimension and the fractal dimension of the global attractor $\ms{A}$ are finite and their upper bounds are given by
$$
	d_{H} (\ms{A}) \leq m \quad \textup{and} \quad d_{F} (\ms{A}) \leq 2m,
$$
where $m$ satisfies \eqref{dimc}. 
\end{proof}

As a conclusion remark, in this work it is shown that global dynamics of this extended Brusselator system of cubic-autocatalytic, partially reversible reaction-diffusion equations with linear coupling between two compartments and with homogeneous Dirichet boundary condition is dissipative and asymptotically determined by a finite-dimensional global attractor in the $H_0^1$ product phase space. 

We emphasize that even though this multi-dimensional reaction-diffusion system lacks the mathematical dissipative condition, the existence of a global attractor is established without assuming any non-negativity of initial data (or solutions) and without any conditions imposed on any of the positive parameters. The contributed methodology of grouping and re-scaling estimation succeeded in proving the dissipative longtime dynamics has a promising potentiality to be adopted and adapted for future investigations of many open problems and complex coupling models in system biology and in network dynamics. 

The extensions of the obtained results can be pursued in several directions. Semicontinuity of the global attractor with respect to one of the diffusive coefficients and/or with respect to the reverse reaction rate coefficient when it tends to zero is a meaningful question of singular perturbation of global dynamics. One can also consider similarly coupled reaction-diffusion systems on a higher-dimensional domain of space dimension $n > 3$ and on an unbounded domain to work with various different phase spaces. 

A challenging problem is to study the two-cell or multi-cell coupling system of the Brusselator equations or the Gray-Scott equations in which one subsystem is defined on a subdomain and the other subsystem is defined on another non-overlapping subdomain with the flux-type coupling through the joint boundary, such as in many biological and physiological models of signal transductions. This kind of reaction-diffusion systems of investigation will be very close to the real-world problems in cell biology and in physical chemistry.



\bibliographystyle{amsplain}

\end{document}